		\documentclass[10pt,leqno]{article}
		
		\usepackage{amsmath,amssymb,amsthm,mathrsfs,dsfont}
		
		\usepackage[margin=3cm]{geometry} %宽版时用，配合双倍行距
		
		\usepackage{titlesec,hyperref}

		\usepackage{color}
		\usepackage{fancyhdr}
		\titleformat{\subsection}{\it}{\thesubsection.\enspace}{1pt}{}
		
		\newtheorem{theo}{Theorem}[section]
		\newtheorem{lemm}[theo]{Lemma}

		\newtheorem{prop}[theo]{Proposition}
		\newtheorem{rema}[theo]{Remark}
		\numberwithin{equation}{section}
		
		\allowdisplaybreaks %允许公式分页显示
		
		\newcommand\lm{{\lesssim}}
\begin{document}
			\title{Large time behavior to a 2D micro-macro model for compressible polymeric fluids near equilibrium
				\hspace{-4mm}
			}
			
			\author{ Wenjie $\mbox{Deng}^1$ \footnote{email: detective2028@qq.com},\quad
				Wei $\mbox{Luo}^1$\footnote{E-mail:  luowei23@mail2.sysu.edu.cn} \quad and\quad
				Zhaoyang $\mbox{Yin}^{1,2}$\footnote{E-mail: mcsyzy@mail.sysu.edu.cn}\\
				$^1\mbox{Department}$ of Mathematics,
				Sun Yat-sen University, Guangzhou 510275, China\\
				$^2\mbox{Faculty}$ of Information Technology,\\ Macau University of Science and Technology, Macau, China}
			
			\date{}
			\maketitle
			\hrule
			
			\begin{abstract}
				In this paper, we mainly study the large time behavior to a 2D micro-macro model for compressible polymeric fluids with small initial data. This model is a coupling of isentropic compressible Navier-Stokes equations with a nonlinear Fokker-Planck equation. Firstly the Fourier splitting method yields that the logarithmic decay rate. By virtue
				of the time weighted energy estimate, we can improve the decay rate to $(1 + t)^{-\frac{1}{4}}$. Under the
				low-frequency condition and by the Littlewood-Paley theory, we show that the solutions belong to
				some Besov spaces with negative index and obtain the optimal $L^2$ decay rate. Finally, we obtain the $\dot{H}^s$ decay rate by establishing a new Fourier splitting estimate.\\
				%Our strategy relies on the results established in our previous work \cite{Zhang-Yin}.
				\vspace*{5pt}
				\noindent {\it 2010 Mathematics Subject Classification}: 35Q30, 76B03, 76D05, 76D99.
				
				\vspace*{5pt}
				\noindent{\it Keywords}: The compressible polymeric fluids; global strong solutions; time decay rate.
			\end{abstract}
			
			\vspace*{10pt}
			
			%\phantomsection
			%\addcontentsline{toc}{section}{\contentsname}
			%添加目录到书签
			\tableofcontents
			
			\section{Introduction}
	   In this paper, we consider a micro-macro model for compressible polymeric fluids near equilibrium with dimension $d\geq2$ \cite{2017Global} :
	\begin{align}\label{eq0}
		\left\{
		\begin{array}{ll}
			\varrho_t+{\rm div}(\varrho u)=0 , \\[1ex]
			(\varrho u)_t+{\rm div}(\varrho u\otimes u)-\rm div\Sigma{(u)}+\frac 1 {Ma^2} \nabla{P(\varrho)}=\frac 1 {De} \frac {\kappa} {r} \rm div~\tau, \\[1ex]
			\psi_t+u\cdot\nabla\psi={\rm div}_{q}[- \nabla u \cdot{q}\psi+\frac {\sigma} {De}\nabla_{q}\psi+\frac {1} {De\cdot r}\nabla_{q}\mathcal{U}\psi],  \\[1ex]
			\tau_{ij}=\int_{\mathbb{R}^{d}}(q_{i}\nabla_{q_j}\mathcal{U})\psi dq, \\[1ex]
			\varrho|_{t=0}=\varrho_0,~~u|_{t=0}=u_0,~~\psi|_{t=0}=\psi_0, \\[1ex]
		\end{array}
		\right.
	\end{align}
	where $\varrho(t,x)$ is the density of the solvent, $u(t,x)$ stands for the velocity of the polymeric liquid and $\psi(t,x,q)$ denotes the distribution function for the internal configuration. Here the polymer elongation $q\in\mathbb{R}^{d}$ , $x\in\mathbb{R}^{d}$ and $t\in[0,\infty)$.
	%where $\Omega$ is a smooth bounded domain of $\mathbb{R}^d$ or $\Omega=\mathbb{R}^d$ or $\Omega=\mathbb{T}^d$. In the case when $\Omega$ has a boundary, we add the boundary condition $u=0$ on $\partial\Omega$.
	The notation 
	$$\Sigma{(u)}=\mu\left(\nabla u+\nabla^{T} u\right)+\mu'{\rm div}u\cdot Id$$
	stands for the stress tensor, with $\mu$ and $\mu'$ being the viscosity coefficients satisfying the relation $\mu>0$ and $2\mu+\mu'>0$.
	The pressure obeys the so-called $\gamma$-law: $$P(\varrho)=a\varrho^\gamma~~~~~~~~$$ 
	with $\gamma\geq1, a>0$. $\sigma$ is a constant satisfing the relation $\sigma=k_BT_a$, where $k_B$ is the Boltzmann constant and $T_a$ is the absolute temperature.  Furthermore, $r>0$ is related to
	the linear damping mechanism in dynamics of the microscopic variable $q$, and $\kappa > 0$ is some parameter describing the ratio between kinetic and elastic energy. The parameter $De$ denotes the Deborah number, which represents the ratio of the time scales for elastic stress relaxation, so it characterizes the fluidity of the system. The Mach number $Ma$ describes the ratio between the fluid velocity and the sound speed, which measures the compressibility of the system.
	Moreover, the potential $\mathcal{U}(q)$ follows the same assumptions as that of \cite{2017Global} :
	\begin{align}\label{potential1}
		\left\{
		\begin{array}{ll}
			|q|\lesssim(1+|\nabla_q\mathcal{U}|),\\
			\Delta_q\mathcal{U}\leq C + \delta|\nabla_q\mathcal{U}|^2,\\
			\int_{\mathbb{R}^{d}}|\nabla_q\mathcal{U}|^2\psi_{\infty}dq\leq C,~~\int_{\mathbb{R}^{d}}|q|^4\psi_{\infty}dq\leq C,
		\end{array}
		\right.
	\end{align}
	with $\delta\in(0,1)$ and
	\begin{align}\label{potential2}
		\left\{
		\begin{array}{ll}
			|\nabla^k_q(q\nabla_q\mathcal{U})|\lesssim(1+|q||\nabla_q\mathcal{U}|),\\
			\int_{\mathbb{R}^d}|\nabla^k_q(q\nabla_q\mathcal{U}\sqrt{\psi_{\infty}})|^2dq\leq C,\\
			|\nabla^k_q(\Delta_q\mathcal{U}-\frac{1}{2}|\nabla_q\mathcal{U}|^2)|\lesssim(1+|\nabla_q\mathcal{U}|^2),~~~~~
		\end{array}
		\right.
	\end{align}
	with the integer $k\in[1,3]$.  The most simplest case among them is the Hookean spring $\mathcal{U}(q)=\frac{1}{2}|q|^2$. From the Fokker-Planck equation, one can derive the following compressible Oldroyd-B equation, which has been studied in depth in \cite{incomprelimDFang,incomprelimZLei,compre1,compre2,2018Global} :
	\begin{align}\label{tauequ}
		\tau_t + u\cdot\nabla\tau + 2\tau  + Q(\nabla u,\tau) = Du~,~~~~~~(Du)^{ij}=\mathop{\sum}\limits_{l,m=1}^d\nabla^lu^m\int_{\mathbb{R}^d}q^lq^mq^iq^j e^{-|q|^2}dq~.
	\end{align}
	It is universally  known that the system \eqref{eq0} can be used to described the fluids coupling polymers. The system is of great interest in many branches of physics, chemistry, and biology, see \cite{Bird1977,Doi1988}. In this model, a polymer is idealized as an ``elastic dumbbell" consisting of two ``beads" joined by a spring that can be modeled by a vector $q$. The polymer particles are described by a probability function $\psi(t,x,q)$ satisfying that $\int_{\mathbb{R}^{d}}
	\psi(t,x,q)dq =1$, which represents the distribution of particles' elongation vector $q\in \mathbb{R}^{d}$. 
	In comparision to the macro-micro models studied in \cite{2017Global}, it's reasonable for us to remove the term ${\rm div} u \psi$ from the equation $\psi$ obey. Otherwise, one can derive ${\rm div} u =0$ from the equation $\psi$ obey by assumption $\int_{\mathbb{R}^d}\psi dq =\int_{\mathbb{R}^d}\psi_0 dq$. However, the condition $\int_{\mathbb{R}^d}\psi dq =\int_{\mathbb{R}^d}\psi_0 dq$ is of great significance in the estimation of ${\rm div}_q \left(\nabla uq\psi\right)$ and it seems impossible to obtain global priori estimation without any dissipation or conservation law for $\int_{\mathbb{R}^d}\psi dq$. Therefore, we underlined that the system \eqref{eq0} is meaningful and the results in this paper indeed cover those obtained in \cite{2017Global,2018Global} . At the level of liquid, the system couples the Navier-Stokes equation for the fluid velocity with a Fokker-Planck equation describing the evolution of the polymer density. One can refer to  \cite{Bird1977,Doi1988,Masmoudi2008,Masmoudi2013} for more details.

	%1.环境
	%\begin{xxxx}
	%
	%\end{xxxx}
	%
	% 如 document、Theorem、align、itemize ……。
	%
	%2.如自己命名公式编号，用命令\tag{...}。
	%3.{align*}不带编号。
	%4.align 环境中如有多行公式，换行用换行符号 \\ ；如需只显示一个编号，在 \\ 前加命令 \nonumber 。
	%5.单行居中公式亦可用 $$……$$，无编号；行内公式 $……$ 。

	% 666.在环境，如公式 \begin{align} 后加标签 \label{...},下文如用到说对式 1.2 如何如何，就可以引用 \ref{...} , 即显示 1.2。即使前文修改加上其他式子，仍自动显示对应的编号。

		In this paper we will take $a,~\sigma,~\kappa,~r,~De,~Ma$ equal to $1$.
		It is easy to check that $(1,0,\psi_{\infty})$ with $$\psi_{\infty}(q)=\frac{e^{-\mathcal{U}(q)}}{\int_{\mathbb{R}^{d}}e^{-\mathcal{U}(q)}dq}~,~~~~~~~~$$
		is  a stationary solution to the system \eqref{eq0}. Considering the perturbations near the global equilibrium:
		\begin{align*}
			\rho=\varrho-1,~~u=u~~~\text{and}~~~g=\frac {\psi-\psi_\infty} {\psi_\infty}~,~~~~~
		\end{align*}
		then we can rewrite system \eqref{eq0} as follows :
		\begin{align}\label{eq1}
			~~~~~~~~~~~~~~~~~~
			\left\{
			\begin{array}{ll}
				\rho_t+\rm divu(1+\rho)=-u\cdot\nabla\rho , \\[1ex]
				u_t-\frac 1 {1+\rho} \rm div\Sigma{(u)}+\frac {P'(1+\rho)} {1+\rho} \nabla\rho=-u\cdot\nabla u+\frac 1 {1+\rho} \rm div\tau, \\[1ex]
				g_t+\mathcal{L}g=-u\cdot\nabla g-\frac 1 {\psi_\infty}\nabla_q\cdot(\nabla uqg\psi_\infty)-\rm divu-\nabla u q\nabla_{q}\mathcal{U},  \\[1ex]
				\tau_{ij}(g)=\int_{\mathbb{R}^{d}}(q_{i}\nabla_{q_j}\mathcal{U})g\psi_\infty dq, \\[1ex]
				\rho|_{t=0}=\rho_0,~~u|_{t=0}=u_0,~~g|_{t=0}=g_0, \\[1ex]
			\end{array}
			\right.
		\end{align}
		where $\mathcal{L}g=-\frac 1 {\psi_\infty}\nabla_q\cdot(\psi_\infty\nabla_{q}g)$ .
				
				\subsection{Short reviews for the incompressible polymeric fluid models}
				So far, there are lots of results about the incompressible model that had been constructed extensively. M. Renardy \cite{Renardy} established the local well-posedness in Sobolev spaces with potential $\mathcal{U}(q)=(1-|q|^2)^{1-\sigma}$ for $\sigma>1$. Later, B. Jourdain et al. \cite{Jourdain} proved local existence of a stochastic differential equation with potential $\mathcal{U}(q)=-k\log(1-|q|^{2})$ in the case $k>3$ for a Couette flow. P. L. Lions and N. Masmoudi \cite{Lions-Masmoudi} constructed global weak solution for some Oldroyd models in the co-rotation case. In order to give a sufficient condition of non-breakdown for an incompressible viscoelastic fluid of the Oldroyd type, J. Y. Chemin and N. Masmoudi \cite{lifespan} established a new priori estimate for 2D Navier-Stokes system and derived a losing derivative estimate for the transport equation, which is of significance in the proof of the global well-posedness for viscoelastic fluids. Later, N. Masmoudi et al. \cite{Masmoudi2dfluid} obtained global solutions for 2D polymeric fluid models under the co-rotational assumption without any small conditions. In addition, Z. Lei et al. \cite{Z.Leinewmethod} provided a new method to improve the criterion for viscoelastic systems of Oldroyd type considered by \cite{lifespan}. It's noting that this method is much easier and is expected to be adopted to other problems involving the prior estimate of losing derivative.
				
				L. He and P. Zhang \cite{He2009} studied the long time decay of the $L^2$ norm to the incompressible Hooke dumbbell models and found that the solutions tends to the equilibrium by $(1+t)^{-\frac{3}{4}}$ under the low-frequency assumption that $u_0\in L^1(\mathbb{R}^3)$. Recently, M. Schonbek \cite{Schonbekdecay} studied the $L^2$ decay of the velocity $u$ to the co-rotation FENE dumbbell model and proved that velocity $u$ tends to zero in $L^2$ by  $(1+t)^{-\frac{d}{4}+\frac{1}{2}}$, $d\geq 2$ when the initial perturbation is additionally bounded in $L^1\mathbb{R}^d$. More recently, W. Luo and Z. Yin \cite{Luo-Yin,Luo-Yin2} improved Schonbek's result and showed that the optimal decay rate of velocity $u$ in $L^2$ should be $(1+t)^{-\frac{d}{4}}$.
				\subsection{Short reviews for the compressible polymeric fluid models}
				Z. Lei \cite{incomprelimZLei} first investigated the incompressible limit problem of the compressible Oldroyd-B model in
				torus. Recently, D. Fang and R. Zi \cite{incomprelimDFang} studied the global well-posedness for compressible Oldroyd-B
				model in critical Besov spaces with $d\geq2$. Z. Zhou et al. \cite{2018Global} proved the global well-posedness and time decay rates for the 3D compressible Oldroyd-B model. More details for the compressible Oldroyd-B type model based on the deformation tensor can refer to \cite{compre1,compre2}. Recently, N. Jiang et al. \cite{2017Global} employed the energetic variational method to derive a micro-macro model for compressible polymeric fluids and proved the global existence near the equilibrium in the Sobolev space $H^3(\mathbb{R}^3)$. W. Deng et al. \cite{decay2022} established the global well-posedness for a micro-macro model for compressible polymeric fluids near equilibrium in Sobolev spaces with $d\geq2$ and obtained the optimal time decay rates in $L^2(\mathbb{R}^d)$ and $\dot{H}^s(\mathbb{R}^d)$ with $d\geq3$ when the initial perturbation is additionally bounded in $\dot{B}^{-\frac{d}{2}}_{2,\infty}(\mathbb{R}^d)$. Z. Luo et al. \cite{FENEdecay} studied the global strong solution for the compressible FENE models near equilibrium with $d\geq2$ and obtained the optimal decay rate in $L^2(\mathbb{R}^d)$ under the low-frequency assumption for initial data.
				\subsection{Main results}
				
				The long time behavior for polymeric models is noticed by N. Masmoudi \cite{2016Masmoudi}. To our best knowledge, large time behaviour for the $2D$ compressible polymeric fluid model with general Hookean potentials given by $(\ref{potential1})$, $(\ref{potential2})$ have not been studied yet. In this paper, we are devoted to the study of large time behaviour optimal time decay rate of $(\ref{eq1})$ in $L^2(\mathbb{R}^2)$ and $\dot{H}^{s}(\mathbb{R}^2)$. The brief outline of the proof is carrying out in the following.
				
				Firstly we can prove the logarithmic decay rate for the
				velocity. The main difficulty for us is to get the initial algebraic decay rate for $u$. By virtue of the time weighted energy estimate and the logarithmic decay rate obtained, then we improve the
				initial decay rate to $(1+t)^{-\frac{1}{4}}$ for the velocity in $L^2(\mathbb{R}^2)$. Different from incompressible cases, we can not obtain the estimate of $\|u\|_{L^1}$, which forces us to obtain half decay rate merely at the very beginning. We will use three steps to overcome the difficulty. \\
				{\bf Step 1}. Estimate the $\dot{B}^{-\frac{1}{2}}_{2,\infty}(\mathbb{R}^2)$-norm of the velocity as in \cite{FENEdecay}. \\
				{\bf Step 2}. By the Littlewood-Paley decomposition theory and the Fourier splitting method, we can improve the time decay rate:
				$$\|u\|_{L^2}\leq (1+t)^{-\frac{5}{16}}.$$
				{\bf Step 3}. Estimate the $\dot{B}^{-1}_{2,\infty}(\mathbb{R}^2)$-norm of the velocity, which eventually leads the optimal time decay rate for the velocity by iteration.
				
				It is worthy mentioning that the optimal decay rates of $(\rho,u,g)$ in $\dot H^s(\mathbb{R}^2)$ is absolutely innovative. It fails to obtain the decay rate of $(\rho,u)$ in $\dot{H}^s(\mathbb{R}^2)$ by the same way as that of $L^2(\mathbb{R}^2)$ since
				$$\|\Lambda^s(\rho,u)\|^2_{L^2} + \eta\langle\Lambda^{s-1}u,\Lambda^s\rho\rangle \simeq \|\Lambda^s(\rho,u)\|^2_{L^2}$$
				does not hold anymore for any $\eta>0$. To overcome the difficulty, a critical Fourier splitting estimate is established in Lemma \ref{2lem2}, which implies that the dissipation of $\rho$ only in high frequency fully enables us to obtain optimal decay rate in $\dot{H}^s(\mathbb{R}^2)$. It should be pointed out that the frequency decomposition considered in this paper is in time dependency, which conduces to obtain optimal decay rate without finite induction argument. Moreover, we can obtain the $\frac{1}{2}$ faster algebric decay rate for $g$.
				
				Let's recall the following theorem.
				\begin{theo}\cite{decay2022}\label{th1}
					Let $d \geq 2~and~s>1+\frac{d}{2}$. Assume $(\rho,u,g)$ be a classical solution of \eqref{eq1} with the initial data $(\rho_0,u_0,g_0)$ satisfying the conditions $\int_{\mathbb{R}^d} g_{0}\psi_{\infty}dq=0$ and $1+g_0>0$, then there exists some sufficiently small constant $\varepsilon_0$ such that if
					\begin{align}
						E_{\lambda}(0)=\|\rho_0\|^2_{H^s}+\|u_0\|^2_{H^s}+\|g_0\|^2_{H^s(\mathcal{L}^2)}+\lambda\|\langle q \rangle g_0\|^2_{H^{s-1}(\mathcal{L}^{2})} \leq \varepsilon_0,
					\end{align}
					then the compressible system \eqref{eq1} admits a unique global classical solution $(\rho,u,g)$ with $\int_{\mathbb{R}^d}g\psi_{\infty}dq=0$ and $1+g>0$ and
					\begin{align}\label{ineq0}
						\sup_{t\in[0,+\infty)} E_{\lambda}(t)+\int_{0}^{\infty}D_{\lambda}(t)dt\leq \varepsilon,
					\end{align}
					where $\varepsilon$ is a small constant dependent on the viscosity coefficients.
				\end{theo}
				Our main result can be stated as follows.
				\begin{theo}\label{th2}
					Set $d=2$. Let $(\rho,u,g)$ be a global strong solution of \eqref{eq1} with the initial data $(\rho_0,u_0,g_0)$ under the condition in Theorem \ref{th1}. In addition, if $(\rho_0,u_0,\int_{R^2} q\otimes\nabla_q\mathcal{U}g_0 dq)\in \dot{B}^{-1}_{2,\infty}(\mathbb{R}^2)\times \dot{B}^{-1}_{2,\infty}(\mathbb{R}^2)\times \dot{B}^{-1}_{2,\infty}(\mathbb{R}^2)$, then there exists a constant $C$ such that
					\begin{align}\label{decay}
						\left\{\begin{array}{l}
							\|\Lambda^\sigma(\rho,u)\|_{L^2}\leq C(1+t)^{-\frac 1 2 - \frac \sigma 2},~\sigma\in[0,s],\\
							\|\Lambda^\delta g\|_{L^2(\mathcal{L}^2)}\leq C(1+t)^{-\frac{1}{2}-\frac{1}{2} - \frac \delta 2},~\delta\in[0,s-1],\\
							\|\Lambda^\eta g\|_{L^2(\mathcal{L}^2)}\leq C(1+t)^{-\frac{1}{2} - \frac {s} {2} },~~\eta\in(s-1,s],
						\end{array}\right.
					\end{align}
				\end{theo}
				\begin{rema}
				 One can see that the system \eqref{eq0} includes the Oldroyd-B type equations by taking the classical Hookean spring $\mathcal{U}(q)=\frac{1}{2}|q|^2$. Therefore, Theorems \ref{th2} implies the optimal decay rate for the classical Oldroyd-B model. It is noting that we obtain the $\frac{1}{2}$ faster algebraic decay rate for $\tau$ or $\|g\|_{\mathcal{L}^2}$ in $\dot{H}^\eta$ for any $\eta\in[s-2,s]$ than those the Oldroyd-B type equations can attain originally. To the best of our knowledge, Theorem \ref{th2} is the first result for the highest order derivative
					decay of solutions to the $2D$ micro-macro model for compressible polymeric fluids.
				\end{rema}
				\begin{rema}
					When we derive the dissipation of $\rho$ in high frequency, $\|\Lambda^{s}u^{high}\|^2_{L^2}$ arises from linear term ${\rm div}~u$, which may lead to the loss of time decay rate. In order to overcome the difficulty, we consider the time weight $(1+t)^{-1}$ such that it can be controlled by the dissipation of $u$, i.e. $$(1+t)^{-1}\|\Lambda^{s}u^{high}\|^2_{L^2}\lesssim\|\Lambda^{s+1}u\|^2_{L^2}~.$$
					We point out that the time weight we consider is critial indeed in our proof since it would be failed if we choose the time weight as $(1+t)^{-\alpha}$ for any $\alpha\in(0,1)$. More details refer to Lemma \ref{2lem2}~.
				\end{rema}
				
				The paper is organized as follows. In Section 2 we introduce some notations and give some preliminaries which will be used in the sequel. In Section 3 we study the $L^2$ decay of solutions to a micro-macro model for compressible polymeric fluids near equilibrium with general Hookean potentials by using the Fourier splitting method. In Section 4 we study the $\dot{H}^s$ decay by establishing a new critical Fourier splitting estimate.
				%\vspace*{2em}
				%\noindent\textbf{Acknowledgements}. This work was partially supported by ...
				
				\section{Preliminaries}
			In this section we will introduce some notations and useful lemmas which will be used in the sequel. If the function spaces are over $\mathbb{R}^d$ for the variable $x$ and $q$, for simplicity, we drop $\mathbb{R}^d$ in the notation of function spaces if there is no ambiguity.
			
			For $p\geq1$, we denote by $\mathcal{L}^{p}$ the space
			$$\mathcal{L}^{p}=\big\{f \big|\|f\|^{p}_{\mathcal{L}^{p}}=\int_{\mathbb{R}^d} \psi_{\infty}|f|^{p}dq<\infty\big\}.~~~~~~~~~~~~~~~~$$
			
			We will use the notation $L^{p}_{x}(\mathcal{L}^{q})$ to denote $L^{p}[\mathbb{R}^{d};\mathcal{L}^{q}]:$
			$$L^{p}_{x}(\mathcal{L}^{q})=\big\{f \big|\|f\|_{L^{p}_{x}(\mathcal{L}^{q})}=(\int_{\mathbb{R}^{d}}(\int_{\mathbb{R}^d} \psi_{\infty}|f|^{q}dq)^{\frac{p}{q}}dx)^{\frac{1}{p}}<\infty\big\}.$$
			
			The symbol $\widehat{f}=\mathcal{F}(f)$ denotes the Fourier transform of $f$.
			Let $\Lambda^s f=\mathcal{F}^{-1}(|\xi|^s \widehat{f})$.
			If $s\geq0$, we can denote by $H^{s}(\mathcal{L}^{2})$ the space
			$$~~~~~~H^{s}(\mathcal{L}^{2})=\{f\big| \|f\|^2_{H^{s}(\mathcal{L}^{2})}=\int_{\mathbb{R}^{d}}\int_{\mathbb{R}^{d}}(|f|^2+|\Lambda^s f|^2)\psi_\infty dRdx<\infty\}.$$
			Then we introduce the energy and energy dissipation functionals for the fluctuation $(\rho,u,g)$ as follows:
			\begin{align*}
				~~~~~~~~~~~~~~~E_{\lambda}(t)&=\sum_{m=0,s}\left(\|h(\rho)^{\frac 1 2}\Lambda^m\rho \|^2_{L^2}+\|(1+\rho)^{\frac 1 2}\Lambda^m u\|^2_{L^2}+\|\Lambda^mg\|^2_{L^2(\mathcal{L}^{2})} \right) \\ \notag
				& ~~~~ + \sum_{m=0,s-1}\lambda\|\langle q \rangle \Lambda^m g\|^2_{L^2(\mathcal{L}^{2})}~,
			\end{align*}
			and
			\begin{align*}
				~~~~~~~~~~~~~~~~D_{\lambda}(t)&=\gamma\|\nabla\rho\|^2_{H^{s-1}}+\mu\|\nabla u\|^2_{H^{s}}+(\mu+\mu')\|{\rm div}u\|^2_{H^{s}}+\|\nabla_q g\|^2_{H^{s}(\mathcal{L}^{2})} \\ \notag
				&~~~+ \lambda\|\langle q \rangle \nabla_q g\|^2_{H^{s-1}(\mathcal{L}^{2})}~,~~~~~~~~~~~~~~~~~~~~~~~~~~~~~
			\end{align*}
			where positive constant $\lambda$ is small enough. Sometimes we write $f\lm g$ instead of $f\leq Cg$, where $C$ is a positive constant. We agree that $\nabla$ stands for $\nabla_x$ and $div$ stands for $div_x$.\\
			
			We recall the Littlewood-Paley decomposition theory and and Besov spaces.
			\begin{prop}\cite{Bahouri2011}\label{prop0}
				Let $\mathcal{C}$ be the annulus $\{\xi\in\mathbb{R}^d:\frac 3 4\leq|\xi|\leq\frac 8 3\}$. There exist radial functions $\chi$ and $\varphi$, valued in the interval $[0,1]$, belonging respectively to $\mathcal{D}(B(0,\frac 4 3))$ and $\mathcal{D}(\mathcal{C})$, and such that
				$$ \forall\xi\in\mathbb{R}^d,\ \chi(\xi)+\sum_{j\geq 0}\varphi(2^{-j}\xi)=1, $$
				$$ \forall\xi\in\mathbb{R}^d\backslash\{0\},\ \sum_{j\in\mathbb{Z}}\varphi(2^{-j}\xi)=1, ~~~$$
				$$ |j-j'|\geq 2\Rightarrow\mathrm{Supp}\ \varphi(2^{-j}\cdot)\cap \mathrm{Supp}\ \varphi(2^{-j'}\cdot)=\emptyset, $$
				$$ ~~j\geq 1\Rightarrow\mathrm{Supp}\ \chi(\cdot)\cap \mathrm{Supp}\ \varphi(2^{-j}\cdot)=\emptyset. $$
				The set $\widetilde{\mathcal{C}}=B(0,\frac 2 3)+\mathcal{C}$ is an annulus, then
				$$ |j-j'|\geq 5\Rightarrow 2^{j}\mathcal{C}\cap 2^{j'}\widetilde{\mathcal{C}}=\emptyset. ~~$$
				Further, we have
				$$ \forall\xi\in\mathbb{R}^d,\ \frac 1 2\leq\chi^2(\xi)+\sum_{j\geq 0}\varphi^2(2^{-j}\xi)\leq 1, $$
				$$ \forall\xi\in\mathbb{R}^d\backslash\{0\},\ \frac 1 2\leq\sum_{j\in\mathbb{Z}}\varphi^2(2^{-j}\xi)\leq 1. ~~~~~$$
			\end{prop}
			
			$\mathcal{F}$ represents the Fourier transform and  its inverse is denoted by $\mathcal{F}^{-1}$.
			Let $u$ be a tempered distribution in $\mathcal{S}'(\mathbb{R}^d)$. For all $j\in\mathbb{Z}$ , define
			$$~~~~~~
			\dot\Delta_j u=\mathcal{F}^{-1}\left(\varphi(2^{-j}\cdot)\mathcal{F}u\right)~~~~~~and~~~~~~
			\dot S_j u=\sum_{j'<j}\dot\Delta_{j'}u~.~~~~~~~~~~~~~~~
			$$
			Then the Littlewood-Paley decomposition is given as follows:
			$$ u=\sum_{j\in\mathbb{Z}}\dot\Delta_j u \quad \text{in} \ \ \ \mathcal{S}^{'}_h(\mathbb{R}^d). ~~~~~$$
			Let $s\in\mathbb{R},\ 1\leq p,r\leq\infty.$ The homogeneous Besov space $\dot B^s_{p,r}$ and $\dot B^s_{p,r}(\mathcal{L}^{p'})$ is defined by
			$$ \dot B^s_{p,r}=\{u\in S^{'}_h:\|u\|_{\dot B^s_{p,r}}=\Big\|\left(2^{js}\|\dot\Delta_j u\|_{L^p}\right)_j \Big\|_{l^r(\mathbb{Z})}<\infty\}, ~~~~$$
			$$ \dot B^s_{p,r}(\mathcal{L}^{p'})=\{\phi\in S^{'}_h:\|\phi\|_{\dot B^s_{p,r}(\mathcal{L}^{p'})}=\Big\|\left(2^{js}\|\dot\Delta_j \phi\|_{L_{x}^{p}(\mathcal{L}^{p'})}\right)_j \Big\|_{l^r(\mathbb{Z})}<\infty\}.$$
			
			The following lemma is about the embedding between Lesbesgue and Besov spaces.
			\begin{lemm}\cite{Bahouri2011}\label{lemma1}
				Let $1\leq p\leq 2$ and $d\geq3$. Then it holds that 
				$$L^{p}\hookrightarrow \dot{B}^{\frac{d}{2}-\frac d p}_{2,\infty}.$$
				Moreover, it follows that
				$$L^2=\dot{B}^0_{2,2} .~~~$$
			\end{lemm}
			
			The following lemma is the Gagliardo-Nirenberg inequality.
			\begin{lemm}\cite{Nirenberg}\label{Lemma2}
				Let $d\geq2,~p\in[2,+\infty)$ and $0\leq s,s_1\leq s_2$, then there exists a constant $C$ such that
				$$\|\Lambda^{s}f\|_{L^{p}}\leq C \|\Lambda^{s_1}f\|^{1-\theta}_{L^{2}}\|\Lambda^{s_2} f\|^{\theta}_{L^{2}},$$
				where $0\leq\theta\leq1$ and $\theta$ satisfy
				$$ s+d(\frac 1 2 -\frac 1 p)=s_1 (1-\theta)+\theta s_2.$$
				Note that we require that $0<\theta<1$, $0\leq s_1\leq s$, when $p=\infty$.
			\end{lemm}
			
			The following lemma allows us to estimate the extra stress tensor $\tau$.
			\begin{lemm}\cite{2017Global}\label{Lemma3}
				Assume $g\in H^s(\mathcal{L}^{2})$ with $\int_{\mathbb{R}^{d}} g\psi_\infty dq=0$, then it follows that
				\begin{align*}
					\left\{
					\begin{array}{ll}
						\|\nabla_q\mathcal{U}g\|_{\dot{H}^\sigma(\mathcal{L}^{2})}+\|qg\|_{\dot{H}^\sigma(\mathcal{L}^{2})}\lesssim\|\nabla_qg\|_{\dot{H}^\sigma(\mathcal{L}^{2})},\\
						\|q\nabla_q\mathcal{U}g\|_{\dot{H}^\sigma(\mathcal{L}^{2})}+\||q|^2g\|_{{\dot{H}^\sigma(\mathcal{L}^{2})}}\lesssim\|\langle q\rangle\nabla_qg\|_{\dot{H}^\sigma(\mathcal{L}^{2})},
					\end{array}
					\right.
				\end{align*}
				for any $\sigma\in[0,s]$.
			\end{lemm}
			
			\begin{lemm}\cite{2017Global}\label{Lemma4}
				Assume $g\in H^s(\mathcal{L}^{2})$, then it holds that
				$$|\tau(g)|^2\lesssim \|g\|^2_{\mathcal{L}^{2}}\lesssim\|\nabla_qg\|^2_{\mathcal{L}^{2}} .$$
			\end{lemm}
			
			\begin{lemm}\cite{Moser1966A}\label{Lemma5}
				Let $s\geq 1$. Assume $p_1,...,p_4$ and $p\in (1,\infty)$ with $\frac 1 p =\frac 1 {p_1}+\frac 1 {p_2}=\frac 1 {p_3}+\frac 1 {p_4}$. Then it holds that
				$$\|[\Lambda^s, f]g\|_{L^p}\leq C\left(\|\Lambda^{s}f\|_{L^{p_1}}\|g\|_{L^{p_2}}+\|\nabla f\|_{L^{p_3}}\|\Lambda^{s-1}g\|_{L^{p_4}}\right).~~$$
				Analogously,
				$$\|[\Lambda^s, f]g\|_{L^2(\mathcal{L}^{2})}\leq C\left(\|\Lambda^{s}f\|_{L^2}\|g\|_{L^\infty(\mathcal{L}^{2})}+\|\nabla f\|_{L^\infty}\|\Lambda^{s-1}g\|_{L^2(\mathcal{L}^{2})}\right).$$
			\end{lemm}
				\section{The $L^2$ decay rate}
				
				This section is devoted to investigating the long time behaviour for the 2D micro-macro model for compressible polymeric fluids with general Hookean springs given by $(\ref{potential1})$, $(\ref{potential2})$. We blame the failure to obtaining the optimal decay rate in $L^2$ as that of \cite{Schonbek1985,Luo-Yin,Luo-Yin2} for the additional stress tensor $\tau$ which does not decay fast enough. To deal with this term, we need to use the coupling effect between $\rho$, $u$ and $g$. Different from the classical Hookean potential $\mathcal{U}=\frac{1}{2}|q|^2$, we can not obtain the optimal decay estimate directly due to the lack of the estimate for $\|(\rho,u)\|_{L^1}$. To over the difficulty, we should consider estimating some Besov-norm with negative index instead. We introduce some notations for simplicity. Denote the energy and energy dissipation functionals considered in this chapter as follows :
				\begin{align*}
					&E_0(t)=\|(\rho,u)\|^2_{H^s} + \|g\|^2_{H^{s}(\mathcal{L}^{2})} ~,\\ \notag
					&E_1(t)=\|\Lambda^1(\rho,u)\|^2_{H^{s-1}} + \|\Lambda^1 g\|^2_{H^{s-1}(\mathcal{L}^{2})}~, \\ \notag	&D_0(t)=\gamma\eta\|\nabla\rho\|^2_{H^{s-1}}+\mu\|\nabla u\|^2_{H^{s}}+(\mu+\mu')\|{\rm divu}\|^2_{H^{s}}
					+\|\nabla_q g\|^2_{H^{s}(\mathcal{L}^{2})} ~ ,\\ \notag
					&D_1(t)=\gamma\eta\|\nabla\Lambda^1\rho\|^2_{H^{s-2}}+\mu\|\nabla \Lambda^1u\|^2_{H^{s-1}}+(\mu+\mu')\|{\rm div\Lambda^1u}\|^2_{H^{s-1}}
					+\|\nabla_q \Lambda^1g\|^2_{H^{s-1}(\mathcal{L}^{2})}~.
				\end{align*}
				
				Denote the following domains involved in Fourier splitting method  :
				$$
				S(t) = \left\{\xi: |\xi|^2 \leq C_2\frac{f'(t)}{f(t)}\right\} ~~~~ \text{with}~~~~~~ f(t) = 1+t ~~\text{or}~~ f(t) = \ln^l(e+t)~,~l\in {\mathbb{N}}.~~~
				$$
				
				Denote the following energy and energy dissipation functionals involved in Fourier splitting method:
				\begin{align*}
					~H_0 = \mu\|u\|^2_{H^s} + \eta\gamma \|\rho\|^2_{H^{s-1}}~~~~~~~~\text{and}~~~~~H_1 =  \mu\|\Lambda^1 u\|^2_{H^{s-1}}+\eta\gamma\|\Lambda^1 \rho\|^2_{H^{s-2}}~.
				\end{align*}
				
				Denote two important factors $B_1$ and $B_2$ :
				$$
				B_1 = \int_0^t\|(\rho,u)\|^4_{L^2} ds ~~~~\text{and}~~~~
				B_2 = \int_{0}^{t}\int_{S(t)}|\hat{G}\cdot\bar{\hat{u}}| d\xi ds'~,
				$$
				
				where 
				$$
				~~~G=-u\cdot\nabla u+\frac{\rho}{1+\rho}\left({\rm div}\Sigma(u)+{\rm div} \tau\right)+\left[\gamma-\frac{P^{'}(1+\rho)}{1+\rho}\right]\nabla\rho.
				$$
				
				 Let's recall the following energy estimate.
				\begin{prop}\cite{decay2022}\label{prop1}
					Under the conditions of Theorem \ref{th2}, it holds for $\sigma=0~or~1$ that
					\begin{align}\label{1ineq1}
						\frac {d} {dt} E_\sigma + D_\sigma \leq 0,
					\end{align}
				\end{prop}
				Let's recall the key lemma of time decay estimate as follows.
				\begin{lemm}\cite{decay2022}\label{1lem1}
					Let $d=2$. Assume $(\rho,u,g)$ be a global strong soluition of system \eqref{eq1} with initial data $(\rho_0,u_0,g_0)$ under the conditions in Theorem  \ref{th2} . Then there exists a positive time $T_0$ such that 
					\begin{align}\label{2ineq1}
						\int_{S(t)}\left(|\hat{\rho}|^2+|\hat{u}|^2+\|\hat{g}\|^2_{\mathcal{L}^2}\right)d\xi
						&\lesssim \frac{f'(t)}{f(t)}\left(1+\|(\rho_0,u_0)\|^2_{\dot{B}^{-\frac d 2}_{2,\infty}}+\|g_0\|^2_{\dot{B}^{-\frac d 2}_{2,\infty}(\mathcal{L}^2)}\right) \\ \notag
						&+ \frac{f'(t)}{f(t)}B_1 + B_2 ~,
					\end{align}
					for any $t>T_0$.
				\end{lemm}
				Firstly we can prove the logarithmic decay rate with $d=2$.
				\begin{prop}\label{prop2}
					Under the conditions of Theorem \ref{th2}, then for any $l\in \mathbb{N}^{+}$, there exists a positive constant $C$ such that
					\begin{align}\label{3ineq1}
						E_0 + (e+t)E_1 \leq C\ln^{-l}(e+t).
					\end{align}
				\end{prop}
				\begin{proof}
					Taking $\sigma=0$ in \eqref{1ineq1}, we have
					\begin{align}\label{3ineq2}
						\frac {d} {dt} E_0(t) + D_0(t) \leq 0.
					\end{align}
					Define $S_0(t)=\left\{\xi:|\xi|^2\leq C_2\frac{f'(t)}{f(t)}\right\}$ with $f(t)=\ln^3(e+t)$ and $C_2$ large enough. Applying Schonbek's strategy to \eqref{3ineq2} as that of \cite{Schonbekstra} leads to
					\begin{align}\label{3ineq3}
						\frac {d} {dt} [f(t)E_0(t)] &+C_2f'(t)H_0(t)+f(t)\| g\|^2_{H^{s}(\mathcal{L}^{2})} \\ \notag
						&\lesssim f'(t)\int_{S_0(t)}|\hat{u}|^2+|\hat{\rho}|^2d\xi + f'(t)\|\rho\|^2_{\dot{H}^s}.
					\end{align}
					It follows from Lemma \ref{1lem1} that
					\begin{align}\label{3ineq4}
						\int_{S_0(t)}|\hat{\rho}|^2+|\hat{u}|^2d\xi
						&\lesssim \ln^{-\frac{1}{2}}(e+t) + B_2 \\ \notag
						&\lesssim \ln^{-\frac{1}{2}}(e+t) + \left(\frac{f^{'}(t)}{f(t)}\right)^{\frac{1}{2}}\int_0^t\|G\|_{L^1}\|u\|_{L^2}ds\\ \notag
						&\lesssim \ln^{-\frac{1}{2}}(e+t) + \left(\frac{f^{'}(t)}{f(t)}\right)^{\frac{1}{2}}\int_0^t\|(\rho,u)\|^2_{L^2}\|\nabla(\rho,u,\tau)\|_{H^1}ds\\ \notag
						&\lesssim \ln^{-\frac{1}{2}}(e+t).
					\end{align}
					According to \eqref{3ineq3} and \eqref{3ineq4}, we deduce that
					\begin{align}\label{3ineq5}
						\frac {d} {dt} [\ln^3(e+t)E_0(t)] \lesssim \frac{\ln^{\frac{3}{2}}(e+t)}{e+t} + \|\rho\|^2_{\dot{H}^s},
					\end{align}
					which implies
					\begin{align}\label{3ineq6}
						E_0(t) \lesssim \ln^{-\frac{1}{2}}(e+t).
					\end{align}
					Nextly, we improve the above decay rate by applying inductive argument. Define $S(t)=\left\{\xi:|\xi|^2\leq C_2\frac{f'(t)}{f(t)}\right\}$ with $f(t)=\ln^{l+3}(e+t)$ and $C_2$ large enough, then we temporarily assume
					\begin{align}\label{3ineq7}
						E_0(t) \lesssim \ln^{-\frac{l}{2}}(e+t).
					\end{align}
				   According to \eqref{3ineq7}, we obtain
					\begin{align}\label{3ineq8}
						B_2=\int_0^t\int_{S_0(t)}|\hat{G}||\hat{u}| d\xi ds &\lesssim \left(\frac{f'(t)}{f(t)}\right)^{\frac{1}{2}}\int_0^t\|G\|_{L^1}\|u\|_{L^2}ds\\ \notag
						&\lesssim \left(\frac{f'(t)}{f(t)}\right)^{\frac{1}{2}}\left(\int_0^t\ln^{-l}(e+s)ds\right)^{\frac{1}{2}}\\ \notag
						& \lesssim
						\ln^{-\frac{l+1}{2}}(e+t),
					\end{align}
					where we have used the fact that 
					\begin{align}\label{3ineq9}
						\lim_{t\rightarrow\infty}\frac{\int_0^t\ln^{-l}(e+s)ds}{(e+t)\ln^{-l}(e+t)}<+\infty,
					\end{align}
					for any $l\in \mathbb{N}^{+}$. Thus it follows from \eqref{3ineq8} that
					\begin{align}\label{3ineq10}
						\frac {d} {dt} [\ln^{l+2}(e+t)E_0(t)] \lesssim \frac{\ln^{\frac{l+3}{2}}(e+t)}{e+t} + \|\rho\|^2_{\dot{H}^s},
					\end{align}
					which implies
					\begin{align}\label{3ineq11}
						E(t) \lesssim \ln^{-\frac{l+1}{2}}(e+t).
					\end{align}
					By virtue of the inductive argument, we finally deduce that
					\begin{align}\label{3ineq12}
						E_0 \lesssim \ln^{-l}(e+t),
					\end{align}
				for any $l \in \mathbb{N}^{+}$. Further, multiplying \eqref{3ineq2} by $\ln^k(e+t)$, we obtain
							\begin{align}\label{3ineq13}
					\frac {d} {dt} [\ln^k(e+t)E_0(t)] + \ln^l(e+t)D_0(t) &\lesssim \frac{\ln^{k-1}(e+t)}{e+t}E_0(t)\\ \notag
					&\lesssim \frac{\ln^{-2}(e+t)}{e+t},
				\end{align}
			which implies the following enhanced integrability that
			\begin{align}\label{3ineq14}
			\int_0^t\ln^k(e+s)D_0ds \leq C,
		\end{align}
	for any $k \in \mathbb{N}^{+}$. Multiplying \eqref{1ineq1} by $(e+t)\ln^l(e+t)$ and taking $\sigma=1$, we infer that
					\begin{align}\label{3ineq15}
						\frac {d} {dt} [(e+t)\ln^{l}(e+t)E_1(t)] 
						&\lesssim \ln^{l}(e+t)E_1(t) \\ \notag
						&\lesssim \ln^{l}(e+t)D_0(t),
					\end{align}
					which implies by using \eqref{3ineq14} that 
					\begin{align}\label{3ineq16}
						(e+t)E_1 \lesssim \ln^{-l}(e+t).
					\end{align}
					We thus complete the proof of Proposition \ref{prop2}.
				\end{proof}
				By virtue of Proposition \ref{prop2}, we are going to prove the initial algebraic time decay rate.
				\begin{prop}\label{prop3}
					Under the condition in Theorem \ref{th2}, there exists a positive constant $C$ such that
					\begin{align}\label{4ineq0}
						E_0 + (1+t)E_1 \leq C(1+t)^{-\frac{1}{2}}.
					\end{align}
				\end{prop}
				\begin{proof}
					Define $S(t)=\{\xi:|\xi|^2\leq\frac{C_2}{1+t}\}$ for some $C_2$ large enough. Taking $\sigma=0$ in $(\ref{1ineq1})$, then we infer from Schonbek's strategy that
					\begin{align}\label{4ineq1}
						\frac {d} {dt} E_0 +\frac{C_2}{(1+t)}\int|\hat{u}|^2+|\hat{\rho}|^2d\xi &+
						\|\Lambda^s\rho\|_{L^2}+  \|\Lambda^{s+1}u\|^2_{L^2} + \| g\|^2_{H^{s}(\mathcal{L}^{2})} \\ \notag
						&\lesssim \frac{1}{1+t}\int_{S(t)}|\hat{u}|^2+|\hat{\rho}|^2d\xi.
					\end{align}
				It follows from Lemma \ref{1lem1} that
				\begin{align}\label{4ineq2}
					\int_{S_0(t)}|\hat{\rho}|^2+|\hat{u}|^2d\xi
					&\lesssim (1+t)^{-1} +(1+t)^{-1}B_1 +B_2 \\ \notag
					&\lesssim (1+t)^{-1} +(1+t)^{-1}\int_0^t\|(\rho,u)\|^4_{L^2}ds\\ \notag
					&~~~~+(1+t)^{-\frac{1}{2}}\int_0^t\|(\rho,u)\|^2_{L^2}\|\nabla(\rho,u,\tau)\|_{H^1}ds.
				\end{align}
				Combining with estimates \eqref{4ineq1} and \eqref{4ineq2}, we deduce that
					\begin{align}\label{4ineq3}
						\frac {d} {dt} E_0 &+\frac{C_2}{(1+t)}\int_{\mathbb{R}^2}|\hat{u}|^2+|\hat{\rho}|^2d\xi+
						\|\Lambda^s\rho\|_{L^2} +  \|\Lambda^{s+1}u\|^2_{L^2} + \| g\|^2_{H^{s}(\mathcal{L}^{2})} \\ \notag
						&\lesssim (1+t)^{-2} + (1+t)^{-2}\int_0^t\|(\rho,u)\|^4_{L^2}ds +(1+t)^{-\frac{3}{2}}\int_0^t\|(\rho,u)\|^2_{L^2}\|\nabla(\rho,u,\tau)\|_{H^1}ds,
					\end{align}
	Multiplying \eqref{4ineq3} by $(1+t)^{2}$ and integrating it over $[0,t]$ leads to
					\begin{align}\label{4ineq4}
						(1+t)^{2}E_0 &\lesssim (1+t)^{1}+(1+t)^{1}\int_0^t\|(\rho,u)\|^4_{L^2}ds \\ \notag
						&~~~~+ (1+t)^{\frac{3}{2}}\int_0^t\|(\rho,u)\|^2_{L^2}\|\nabla(\rho,u,\tau)\|_{H^1}ds.
					\end{align}
					Define $M(t)=\mathop{\sup}\limits_{s\in[0,t]} (1+s)^{\frac{1}{2}}E_0(s)$. According to Proposition $\ref{prop3}$ with $l=2$, we obtain
					\begin{align}\label{4ineq5}
						M(t) &\lesssim  1 + \int_0^t(1+s)^{-1}\|(\rho,u,\tau)\|^2_{L^2}M(s)ds+ \int_0^t(1+s)^{-\frac{1}{2}}\|\nabla(\rho,u,\tau)\|_{H^1}M(s) ds \\ \notag
						&\lesssim  1 + \int_0^t\frac{\ln^{-2}(e+s)}{1+s}M(s)ds.
					\end{align}
					We infer from Gr\"{o}nwall's inequality that $M(t) \lesssim 1$, which implies
					\begin{align}\label{4ineq6}
						E_0 \lesssim (1+t)^{-\frac{1}{2}}.
					\end{align}
				Further, multiplying \eqref{3ineq2} by $(1+t)^{\frac{3}{2}}$, we obtain
				\begin{align}\label{4ineq7}
					\frac {d} {dt} [(1+t)^{\frac{3}{2}}E_0(t)] + (1+t)^{\frac{3}{2}}D_0(t) &\lesssim (1+t)^{\frac{1}{2}}E_0(t),
				\end{align}
				which implies the following enhanced integrability that
				\begin{align}\label{4ineq8}
					(1+t)^{-1}\int_0^t(1+t)^{\frac{3}{2}}D_0ds \leq C.
				\end{align}
			Multiplying \eqref{1ineq1} by $(1+t)^{\frac{5}{2}}$ and taking $\sigma=1$, we infer that
			\begin{align}\label{4ineq9}
				\frac {d} {dt} [(1+t)^{\frac{5}{2}}E_1(t)] 
				&\lesssim (1+t)^{\frac{3}{2}}E_1(t) \\ \notag
				&\lesssim (1+t)^{\frac{3}{2}}D_0(t),
			\end{align}
			which implies by using \eqref{4ineq8} that
					\begin{align}\label{4ineq10}
			(1+t)E_1(t) &\lesssim (1+t)^{-\frac{3}{2}}\int_0^t(1+s)^{\frac{3}{2}}D_0(s)ds \\ \notag
			& \lesssim (1+t)^{-\frac{1}{2}}.
			\end{align}
			 We thus complete the proof of Proposition \ref{prop3}.
				\end{proof}
				By Proposition \ref{prop3}, we can show that the solution of \eqref{eq1} belongs to some Besov spaces with negative index.
				\begin{lemm}\label{1lem2}
					Let $0 < \alpha, \sigma \leq 1$ and $ \sigma < 2\alpha$. Assume that $(\rho_0, u_0, g_0)$ satisfies the condition in Theorem \ref{th2}. If
					\begin{align}\label{5ineq0}
						E_0(t) + (1+t)E_1(t) \leq C(1 + t)^{-\alpha},
					\end{align}
					then it holds that
					\begin{align}\label{5ineq1}
						(\rho,u,\tau)\in L^{\infty}(0,\infty;\dot{B}^{-\sigma}_{2,\infty}).
					\end{align}
				\end{lemm}
				\begin{proof}
					Applying $\dot{\Delta}_j$ to $(\ref{eq1})$, we obtain by virtue of standard energy estimate that
					\begin{align}\label{5ineq2}
						&\frac{1}{2}\frac{d}{dt}(\|\Delta_jg\|^2_{L^2(\mathcal{L}^2)}+\|\Delta_j(\rho,u)\|^2_{L^2}+\eta\langle\Delta_ju,\nabla\Delta_j\rho\rangle)\\ \notag
						&+\eta2^{2j}\|\Delta_j\rho\|^2_{L^2}+2^{2j}\|\Delta_ju\|^2_{L^2}+\|\nabla_g\Delta_jg\|^2_{L^2(\mathcal{L}^2)}\\ \notag &\leq \|\Delta_j(\rho u)\|^2_{L^2}+2^{2j}\|\Delta_j(\rho u)\|_{L^2}\|\Delta_ju\|_{L^2}+ \|\Delta_j(u\cdot\nabla u)\|_{L^2}\|\Delta_ju\|_{L^2}+\|\Delta_j(u\cdot\nabla u)\|^2_{L^2} \\ \notag &+\|\Delta_jg\|^2_{L^2}+\|\Delta_jg\|_{L^2}\|\Delta_ju\|_{L^2}+\|\Delta_j(u\cdot\nabla g)\|^2_{L^2(\mathcal{L}^2)}+\|\Delta_j(\nabla uqg)\|^2_{L^2(\mathcal{L}^2)}.
					\end{align}
					Multiplying both sides of \eqref{5ineq2} by $2^{-2j\sigma}$ and taking $l^{\infty}$ norm for $j\in \mathbb{N}$, we obtain
					\begin{align}\label{5ineq3}
						&\frac{1}{2}\frac{d}{dt}(\|g\|^2_{\dot{B}^{-\sigma}_{2,\infty}(\mathcal{L}^2)}+\|\rho\|^2_{\dot{B}^{-\sigma}_{2,\infty}}+\|u\|^2_{\dot{B}^{-\sigma}_{2,\infty}})+\eta\|\rho\|^2_{\dot{B}^{-\sigma+1}_{2,\infty}}+\|u\|^2_{\dot{B}^{-\sigma+1}_{2,\infty}}+\|\nabla_gg\|^2_{\dot{B}^{-\sigma}_{2,\infty}(\mathcal{L}^2)}\\ \notag &\leq \|\rho u\|^2_{\dot{B}^{-\sigma}_{2,\infty}}+\|\rho u\|_{\dot{B}^{-\sigma+1}_{2,\infty}}\|u\|_{\dot{B}^{-\sigma+1}_{2,\infty}}+ \|u\cdot\nabla u\|_{\dot{B}^{-\sigma}_{2,\infty}}\|u\|_{\dot{B}^{-\sigma}_{2,\infty}}+\|u\cdot\nabla u\|^2_{\dot{B}^{-\sigma}_{2,\infty}} \\ \notag &+\|g\|^2_{\dot{B}^{-\sigma}_{2,\infty}}+\|g\|_{\dot{B}^{-\sigma}_{2,\infty}}\|u\|_{\dot{B}^{-\sigma}_{2,\infty}}+\|u\cdot\nabla g\|^2_{\dot{B}^{-\sigma}_{2,\infty}(\mathcal{L}^2)}+\|\nabla uqg\|^2_{\dot{B}^{-\sigma}_{2,\infty}(\mathcal{L}^2)}.
					\end{align}
					According to Lemma \ref{Lemma4} and \eqref{5ineq0}, we infer
					\begin{align}\label{5ineq4}
						\|\rho u\|^2_{\dot{B}^{-\sigma}_{2,\infty}} \leq \|\rho u\|^2_{L^{\frac{2}{\sigma+1}}} \lesssim \|\rho\|^2_{L^2}\|u\|^2_{L^{\frac{2}{\sigma}}}\lesssim \|\rho\|^2_{L^2}\|u\|^{2\sigma}_{L^2}\|\nabla u\|^{2-2\sigma}_{L^2}\lesssim (1+t)^{-2\alpha-1+\sigma},
					\end{align}
					and
					\begin{align}\label{5ineq5}
						\|u\cdot\nabla u\|^2_{\dot{B}^{-\sigma}_{2,\infty}} \leq \|u\cdot\nabla u\|^2_{L^{\frac{2}{\sigma+1}}} \lesssim \|\nabla u\|^2_{L^2}\|u\|^2_{L^{\frac{2}{\sigma}}}\lesssim \|u\|^{2\sigma}_{L^2}\|\nabla u\|^{4-2\sigma}_{L^2}\lesssim (1+t)^{-2\alpha-2+\sigma},
					\end{align}
					as well as
					\begin{align}\label{5ineq6}
						\|\rho u\|^2_{\dot{B}^{-\sigma+1}_{2,\infty}} \leq \|\rho u\|^2_{L^{\frac{2}{\sigma}}}\lesssim \|(\rho, u)\|^{2}_{H^{s-1}}\|(\rho, u)\|^{2\sigma}_{L^2}\|\nabla(\rho, u)\|^{2-2\sigma}_{L^2}\lesssim (1+t)^{-2\alpha-2+\sigma}.
					\end{align}
					Similarly, by virtue of Lemmas \ref{Lemma4}, \ref{Lemma3} and condition \eqref{5ineq0}, we have
					\begin{align}\label{5ineq7}
						\|u\cdot\nabla g\|^2_{\dot{B}^{-\sigma}_{2,\infty}(\mathcal{L}^2)} \leq \|u\cdot\nabla g\|^2_{L^{\frac{2}{\sigma+1}}(\mathcal{L}^2)} \lesssim \|\nabla g\|^2_{L^2(\mathcal{L}^2)}\|u\|^{2\sigma}_{L^2}\|\nabla u\|^{2-2\sigma}_{L^2}\lesssim (1+t)^{-2\alpha-2+\sigma} ,
					\end{align}
					and
					\begin{align}\label{5ineq8}
						\|g\|^2_{\dot{B}^{-\sigma}_{2,\infty}} \leq \|g\|^2_{L^{\frac{2}{\sigma+1}}}\lesssim \|\nabla(\rho,u,\tau)\|^{2}_{H^{s-1}}\|\rho\|^{2\sigma}_{L^2}\|\nabla\rho\|^{2-2\sigma}_{L^2}\lesssim (1+t)^{-2\alpha-4+\sigma} ,
					\end{align}
					as well as
					\begin{align}\label{5ineq9}
						\|\nabla uqg\|^2_{\dot{B}^{-\sigma}_{2,\infty}(\mathcal{L}^2)} &\leq \|\nabla uqg\|^2_{L^{\frac{2}{\sigma+1}}(\mathcal{L}^2)}\\ \notag
						&\lesssim \|\nabla u\|^{2}_{L^2}\|qg\|^2_{L^{\frac{2}{\sigma}}(\mathcal{L}^2)} + \| qg\|^{2}_{L^2(\mathcal{L}^2)}\|\nabla u\|^2_{L^{\frac{2}{\sigma}}}\\ \notag
						& \lesssim \|\nabla u\|^2_{L^2} + \|\nabla_qg\|^2_{L^2}.
					\end{align}
					Therefore, according to estimates from \eqref{5ineq3} to \eqref{5ineq9}, we conclude that
					\begin{align}\label{5ineq10}
						\|g\|^2_{\dot{B}^{-\sigma}_{2,\infty}(\mathcal{L}^2)}&+\|\rho\|^2_{\dot{B}^{-\sigma}_{2,\infty}}+\|u\|^2_{\dot{B}^{-\sigma}_{2,\infty}}\\ \notag
						&\lesssim 1 + \int_0^t (1+t)^{-(\alpha+1-\frac{\delta}{2})}ds\lesssim 1.
					\end{align}
					This completes the proof of Lemma \ref{1lem2}~.
				\end{proof}
				
				Thanks to Lemma \ref{1lem2}~, we can improve the time decay rate by Littlewood-Paley decomposition theory and Fourier splitting method.
				\begin{lemm}\label{1lem3}
					Let $0 < \beta, \sigma \leq 1$ and $\frac{1}{2} \leq \alpha$. Assume that $(\rho_0, u_0, \tau_0)$ satisfies the condition in
					Theorem \ref{th2}. For any $t\in [0,\infty)$, if
					\begin{align}\label{6ineq0}
						E_0(t) +(1+t)E_1(t) \lesssim (1 + t)^{-\alpha},
					\end{align}
					and
					\begin{align}\label{6ineq1}
						(\rho, u, \tau) \in L^{\infty}(0,\infty; B^{-\sigma}_{2,\infty}),
					\end{align}
					then there exists a positive constant $C$ such that
					\begin{align}\label{6ineq2}
						E_0(t) + (1+t)E_1(t) \leq  C(1 + t)^{-\beta},
					\end{align}
					where $\beta <\frac{\sigma+1}{2}$ for $\alpha = \frac{1}{2}$ and $\beta = \frac{\sigma+1}{2}$ for $\alpha > 1$.
				\end{lemm}
				\begin{proof}
					By virtue of conditions \eqref{6ineq0} and \eqref{6ineq1}, we infer that
					\begin{align}\label{6ineq3}
						(1+t)^{-1}B_1 =  (1+t)^{-1}\int_0^t\|(\rho,u)\|^4_{L^2} ds 
						& \lesssim (1+t)^{-1}\int_0^t(1+t)^{-2\alpha}ds \\ \notag
						&\lesssim (1 + t)^{-\beta},
					\end{align}
					and
					\begin{align}\label{6ineq4}
						B_2 = \int_{S(t)} \int_0^t|\hat{G}||\hat{u}| dsd\xi &\lesssim \int_0^t\|G\|_{L^1}\mathop{\sum}\limits_{j\leq\log_2[\frac{4}{3}C_2^{\frac{1}{2}}(1+t)^{-\frac{1}{2}}]} \int_{S(t)}2^{\sigma j}2^{-\sigma j}\varphi^2(2^{-j}\xi)|\hat{u}|d\xi ds\\ \notag
						& \lesssim (1+t)^{-\frac{1}{2}-\frac{\sigma}{2}}\int_0^t\|u\|_{\dot{B}^{-\sigma}_{2,\infty}}\|(\rho,u)\|_{L^2}\|\nabla(\rho,u,\tau)\|_{H^1}ds\\ \notag
						& \lesssim (1+t)^{-\frac{1}{2}-\frac{\sigma}{2}}\int_0^t(1+t)^{-\alpha-\frac{1}{2}}ds \\ \notag
						&\lesssim (1 + t)^{-\beta},
					\end{align}
					According to the proof of Proposition \ref{prop2} and Lemma \ref{1lem1}, we deduce that
					\begin{align}\label{6ineq5}
						\frac {d} {dt} E_0+\frac{C_2}{(1+t)}\int|\hat{u}|^2+|\hat{\rho}|^2d\xi&+
						\|\Lambda^s\rho\|_{L^2}+  \|\Lambda^{s+1}u\|^2_{L^2} + \| g\|^2_{H^{s}(\mathcal{L}^{2})} \\ \notag
						&\lesssim (1+t)^{-\beta-1}.
					\end{align}
					which implies
					\begin{align}\label{6ineq6}
						E_0(t) \lesssim (1 + t)^{-\beta}.~~
					\end{align}
				By performing a routine procedure, one can arrive at
					\begin{align}\label{6ineq7}
						E_1(t) \lesssim (1 + t)^{-\beta-1}.
					\end{align}
				We thus complete the proof of Lemma \ref{1lem3}.
				\end{proof}
				Thus we can obtain the optimal decay rate in $L^2$ by using the bootstrap argument as follows.
				\begin{prop}\label{prop4}
					Assume that $(\rho_0, u_0, \tau_0)$ satisfies the condition in
					Theorem \ref{th2}, then
					\begin{align}\label{7ineq0}
						E_0(t) + (1 + t)E_1(t) \lesssim (1 + t)^{-1}.~
					\end{align}
					\begin{proof}
						According to Proposition \ref{prop2} and Lemma \ref{1lem2} with
						$\alpha = \sigma = \frac{1}{2}$, we have
						\begin{align}\label{7ineq1}
							(\rho,u) \in L^{\infty}(0, \infty; B^{-\frac{1}{2}}_{2,\infty}),~~~g\in L^{\infty}(0, \infty; B^{-\frac{1}{2}}_{2,\infty}(\mathcal{L}^2)).
						\end{align}
						Taking advantage of Lemma \ref{1lem3} with $\alpha = \sigma = \frac{1}{2}$ and $\beta = \frac{5}{8}$, we deduce that
						\begin{align}\label{7ineq2}
							E_0(t) \lesssim (1 + t)^{-\frac{5}{8}},~~~E_1(t) \lesssim (1 + t)^{-\frac{13}{8}}.
						\end{align}
						Taking $\sigma = 1$ and $\alpha = \frac{5}{8}$
						in Lemma \ref{1lem2}, we infer that
						\begin{align}\label{7ineq3}
							(\rho,u) \in L^{\infty}(0, \infty; B^{-1}_{2,\infty}),~~~g\in L^{\infty}(0, \infty; B^{-1}_{2,\infty}(\mathcal{L}^2)).
						\end{align}
						Using Propositions \ref{1lem3} again with $\alpha = \frac{5}{8}$
						and $\sigma = \beta = 1$, we finally obtain
						\begin{align}\label{7ineq4}
							E_0(t) \lesssim (1 + t)^{-1},~~~E_1(t) \lesssim (1 + t)^{-2} .~~
						\end{align}
					This completes the proof of Proposition \ref{prop4}.
					\end{proof}
				\end{prop}
				\section{The $\dot{H}^s$ decay rate}
				In this section, we consider the optimal decay rate of $(\rho,u)$ in $\dot H^s$ and $g$ in $\dot H^s(\mathcal{L}^2)$. To the best of our knowledge, the dissipation of $\rho$ is of great significance to obtain the decay rate of $(\rho,u)$. However, it fails to obtain the optimal decay rate of $(\rho,u)$ in $\dot{H}^s$ by the same way as that of $L^2$ since the equivalence
				$$\|\Lambda^s(\rho,u)\|^2_{L^2} + \eta\langle\Lambda^{s-1}u,\Lambda^s\rho\rangle \simeq \|\Lambda^s(\rho,u)\|^2_{L^2}$$
				does not hold anymore for any $\eta>0$. In the usual way, we can show that
				$\|\Lambda^s(\rho,u)\|^2_{L^2}$ has the same decay rate as the quantity $\|\Lambda^{s-1}(\rho,u)\|^2_{L^2}$. Motivated by \cite{decay2022}, however, we aware that the dissipation of $\rho$ only in high frequency fully enables us to obtain optimal decay rate in $\dot{H}^s$. More precisely, to overcome the difficulty above, $\langle\Lambda^{s}\dot{\Delta}_j\rho,\Lambda^{s-1} \dot{\Delta}_ju\rangle$ in high frequency $S^c(t)$ is considered and it results in dissipation $\int_{S^c(t)}|\xi|^{2s}|\widehat{\rho}|^2d\xi$. To make full use of the benefit the dissipation $\rho$ provides, a critical fourier splitting estimate is established in the following lemma, which helps to obtain the optimal decay rate in $\dot H^s$. Define $S(t)=\{\xi\in\mathbb{R}^d||\xi|^2\leq\frac{C_2}{1+t}\}$ and consider the frequency decomposition as follows :
				\begin{align}
					\left\{\begin{array}{l}
						\rho=\mathcal{F}^{-1}(\chi_{S(t)}\hat{\rho}) + \mathcal{F}^{-1}(\chi_{S^c(t)}\hat{\rho}) = \rho^{low} + \rho^{high},\\
						u=\mathcal{F}^{-1}(\chi_{S(t)}\hat{u}) + \mathcal{F}^{-1}(\chi_{S^c(t)}\hat{u}) = u^{low} + u^{high}.
					\end{array}\right.
				\end{align}
				
				Then we have the following Lemma.
				\begin{lemm}\label{2lem1}
					Assume that $(\rho_0, u_0, g_0)$ satisfies the condition in
					Theorem \ref{th2}, then there exists a positive constant $\delta$ such that
					\begin{align}\label{8ineq0}
						&\|\Lambda^{s-1}u\|_{L^{4}}\|\nabla u\|_{L^{4}}\|\Lambda^{s+1}u\|_{L^2}~,~\|\Lambda^{s}u\|_{L^{4}}\|\nabla g\|_{L^4}\|\Lambda^{s}g\|_{L^2}\\ \notag
						& \lesssim \varepsilon\left(\|\Lambda^{s+1} u\|^{2}_{L^2}+\|\Lambda^sg\|^2_{L^2(\mathcal{L}^2)}\right) + (1+t)^{-1-\delta}\|\Lambda^{s}u^{low}\|^{2}_{L^2},~~~~~~~~~~~~~~~~~~~~~~~~~
					\end{align}
					and
					\begin{align}\label{8ineq1}
						&\|\nabla u\|_{L^{\infty}}\|\Lambda^{s}\rho\|^2_{L^2} ~,~\|\rho\|^2_{L^{\infty}}\|\Lambda^{s}\rho\|^2_{L^2}~,~\|\nabla\rho\|_{L^{4}}\|\Lambda^su\|_{L^{4}}\|\Lambda^{s}\rho\|_{L^2}~,\\ \notag
						&\|\Lambda^{s-1}\rho\|_{L^{4}}\|\nabla^2u\|_{L^{4}}\|\Lambda^{s+1}u\|_{L^2}~,~\|\Lambda^{s-1}\rho\|_{L^{4}}\|\nabla(\rho,\tau)\|_{L^{4}}\|\Lambda^{s+1}u\|_{L^2}\\ \notag
						& \lesssim \varepsilon\left(\|\Lambda^{s+1} u\|^{2}_{L^2}+\|\Lambda^sg\|^2_{L^2(\mathcal{L}^2)}\right) + (1+t)^{-1-\delta}\left( \|\Lambda^{s}\rho^{high}\|^{2}_{L^2}+\|\Lambda^{s}\rho^{low}\|^{2}_{L^2} \right).
					\end{align}
				\end{lemm}
			\begin{proof}
				According to Lemma \ref{Lemma2} and Proposition \ref{prop4}, we deduce that
				\begin{align}\label{8ineq2}
					\|\rho\|^2_{L^{\infty}}\|\Lambda^{s}\rho\|^2_{L^2} &\lesssim \|\rho\|^{2-\frac{2}{s+1}}_{L^2}\|\Lambda^{s}\rho\|^{2+\frac{2}{s+1}}_{L^2}\\ \notag
					& \lesssim (1+t)^{-1-\frac{1}{s+1}}\left(\|\Lambda^{s}\rho^{high}\|^{2}_{L^2}+\|\Lambda^{s}\rho^{low}\|^{2}_{L^2}\right),~~~~~~~~~~~~~~~~~~~~~
				\end{align}
				and
				\begin{align}\label{8ineq3}
					\|\nabla u\|_{L^{\infty}}\|\Lambda^{s}\rho\|^2_{L^2} &\lesssim \|\nabla u\|^{1-\frac{1}{s}}_{L^2}\|\Lambda^{s+1} u\|^{\frac{1}{s}}_{L^2}\|\Lambda^{s}\rho\|^{\frac{2s-1}{s}+\frac{1}{s}}_{L^2}\\ \notag
					& \lesssim \varepsilon\|\Lambda^{s+1} u\|^{2}_{L^2} + \|\nabla u\|^{\frac{2(s-1)}{2s-1}}_{L^2}\|\Lambda^{s}\rho\|^{\frac{2}{2s-1}+2}_{L^2}\\ \notag
					& \lesssim (1+t)^{-1-\frac{1}{2s-1}}\left(\|\Lambda^{s}\rho^{high}\|^{2}_{L^2}+\|\Lambda^{s}\rho^{low}\|^{2}_{L^2}\right)  + \varepsilon\|\Lambda^{s+1} u\|^{2}_{L^2}.
				\end{align}
				By virtue of Lemmas \ref{Lemma2}, \ref{Lemma4} and Proposition \ref{prop4}, we infer that
				\begin{align}\label{8ineq4}
					\|\Lambda^{s-1}\rho\|_{L^{4}}\|\nabla\tau\|_{L^{4}}\|\Lambda^{s+1}u\|_{L^2}
					&\lesssim \|\nabla\rho\|^{\frac{1}{2(s-1)}}_{L^{2}}\|\Lambda^s\rho\|^{\frac{2s-3}{2(s-1)}}_{L^{2}}\|\nabla\tau\|^{\frac{2s-3}{2(s-1)}}_{L^{2}}\|\Lambda^s\tau\|^{\frac{1}{2(s-1)}}_{L^{2}}\|\Lambda^{s+1}u\|_{L^2} \\ \notag
					& \lesssim \varepsilon\left(\|\Lambda^{s+1}u\|^2_{L^2}+\|\Lambda^sg\|^2_{L^2(\mathcal{L}^2)}\right) + \|\nabla\rho\|^{\frac{2}{2s-3}}_{L^{2}}\|\nabla g\|^{2}_{L^{2}(\mathcal{L}^2)}\|\Lambda^s\rho\|^{2}_{L^{2}} \\ \notag
					& \lesssim (1+t)^{-2-\frac{2}{2s-3}}\left(\|\Lambda^{s}\rho^{high}\|^{2}_{L^2}+\|\Lambda^{s}\rho^{low}\|^{2}_{L^2}\right) \\ \notag
					&~~~~+ \varepsilon\left(\|\Lambda^{s+1}u\|^2_{L^2}+\|\Lambda^s\nabla_qg\|^2_{L^2(\mathcal{L}^2)}\right),
				\end{align}
				and
				\begin{align}\label{8ineq5}
					\|\nabla\rho\|_{L^{4}}\|\Lambda^su\|_{L^{4}}\|\Lambda^{s}\rho\|_{L^2} &\lesssim \|\nabla\rho\|^{\frac{2s-3}{2(s-1)}}_{L^{2}}\|\Lambda^s\rho\|^{\frac{2s+3}{2(s+1)}+\frac{1}{s^2-1}}_{L^{2}}\| u\|^{\frac{1}{2(s+1)}}_{L^{2}}\|\Lambda^{s+1}u\|^{\frac{2s+1}{2(s+1)}}_{L^{2}}\\ \notag
					&\lesssim \varepsilon\|\Lambda^{s+1} u\|^{2}_{L^2} + \|\nabla\rho\|^{\frac{2(2s-3)(s+1)}{(2s+3)(s-1)}}_{L^{2}}\| u\|^{\frac{2}{2s+3}}_{L^{2}}\|\Lambda^s\rho\|^{\frac{4}{(2s+3)(s-1)}+2}_{L^{2}}\\ \notag
					& \lesssim \varepsilon\|\Lambda^{s+1} u\|^{2}_{L^2} + C(1+t)^{-1-\frac{2s}{2s+3}}\left(\|\Lambda^{s}\rho^{high}\|^{2}_{L^2}+\|\Lambda^{s}\rho^{low}\|^{2}_{L^2}\right).~~~~~
				\end{align}
				According to Lemma \ref{Lemma2} and Proposition \ref{prop4}, one can arrive at
				\begin{align}\label{8ineq6}
					\|\Lambda^{s-1}\rho\|_{L^{4}}\|\nabla^2u\|_{L^{4}}\|\Lambda^{s+1}u\|_{L^2}&\lesssim \|\rho\|^{\frac{1}{2s}}_{L^2}\|\Lambda^s\rho\|^{\frac{1}{s}+\frac{2s-3}{2s}}_{L^2}\|\nabla u\|^{\frac{2s-3}{2s}}_{L^2}\|\Lambda^{s+1}u\|^{\frac{2s+3}{2s}}_{L^2}\\ \notag
					&\lesssim \varepsilon\|\Lambda^{s+1}u\|^2_{L^2} + \|\rho\|^{\frac{2}{2s-3}}_{L^2}\|\nabla u\|^{2}_{L^2}\|\Lambda^s\rho\|^{\frac{4}{2s-3}+2}_{L^2}\\ \notag
					& \lesssim \varepsilon\|\Lambda^{s+1} u\|^{2}_{L^2} + C(1+t)^{-1-\frac{2s+2}{2s-3}}\left(\|\Lambda^{s}\rho^{high}\|^{2}_{L^2}+\|\Lambda^{s}\rho^{low}\|^{2}_{L^2}\right),
				\end{align}
			and
				\begin{align}\label{8ineq7}
					\|\Lambda^{s-1}\rho\|_{L^{4}}\|\nabla\rho\|_{L^{4}}\|\Lambda^{s+1}u\|_{L^2}&\lesssim \|\rho\|^{1-\frac{1}{s}}_{L^{2}}\|\Lambda^s\rho\|^{1+\frac{1}{s}}_{L^{2}}\|\Lambda^{s+1}u\|_{L^2} \\ \notag
					& \lesssim \varepsilon\|\Lambda^{s+1}u\|^2_{L^2} + \|\rho\|^{2-\frac{2}{s}}_{L^{2}}\|\Lambda^s\rho\|^{\frac{2}{s}+2}_{L^{2}} \\ \notag
					& \lesssim \varepsilon\|\Lambda^{s+1} u\|^{2}_{L^2} + C(1+t)^{-1-\frac{1}{s}}\left(\|\Lambda^{s}\rho^{high}\|^{2}_{L^2}+\|\Lambda^{s}\rho^{low}\|^{2}_{L^2}\right).~~~~
				\end{align}
			Analogously,
				\begin{align}\label{8ineq8}
					\|\Lambda^{s-1}u\|_{L^{4}}\|\nabla u\|_{L^{4}}\|\Lambda^{s+1}u\|_{L^2} &\lesssim \|\nabla u\|_{L^2}\|\Lambda^su\|_{L^2}\|\Lambda^{s+1}u\|_{L^2}\\ \notag
					& \lesssim \varepsilon\|\Lambda^{s+1} u\|^{2}_{L^2} + C(1+t)^{-2}\|\Lambda^{s}u^{low}\|^{2}_{L^2},~~~~~~~~~~~~~~~~~~~~~~~~~~~~~~~~~~
				\end{align}
				and
				\begin{align}\label{8ineq9}
					\|\Lambda^{s}u\|_{L^{4}}\|\nabla g\|_{L^4}\|\Lambda^{s}g\|_{L^2} &\lesssim \|\nabla g\|^{\frac{2s-3}{2(s-1)}}_{L^2}\|\Lambda^sg\|^{\frac{1}{2(s-1)}+1}_{L^2}\|\Lambda^{s}u\|^{\frac{1}{2}}_{L^{2}}\|\Lambda^{s+1}u\|^{\frac{1}{2}}_{L^2} \\ \notag
					&\lesssim \varepsilon\|\Lambda^{s+1}u\|^2_{L^2}+C\|\nabla g\|^{\frac{4s-6}{3(s-1)}}_{L^2}\|\Lambda^sg\|^{\frac{2}{3(s-1)}+\frac{4}{3}}_{L^2}\|\Lambda^{s}u\|^{\frac{2}{3}}_{L^{2}} \\ \notag
					& \lesssim \varepsilon\left(\|\Lambda^{s+1}u\|^2_{L^2} + \|\Lambda^sg\|^2_{L^2(\mathcal{L}^2)} \right) + C(1+t)^{-\frac{4}{3}}\|\Lambda^{s}u^{low}\|^{2}_{L^2}.~~~~
				\end{align}
				We thus complete the proof of Lemma \ref{2lem1} by taking $\delta=\min\left\{\frac{1}{3},\frac{1}{2s-1},\frac{1}{s+1},\frac{1}{s}\right\}$. 
			\end{proof}
				
				A critical Fourier splitting estimate is established in the following to obtain the optimal time decay rate in the highest order derivative.
				\begin{lemm}\label{2lem2}
					Assume that $(\rho_0, u_0, \tau_0)$ satisfies the condition in
					Theorem \ref{th2} and denote $\sigma_R = \{j\in\mathbb{N}|{\rm supp}\{\varphi(2^{-j}\xi)\}\cap S^c(R)\neq\varnothing\}$, then for any $R\in[0,\infty)$ and $\eta$ sufficiently small, it holds that
					\begin{align}\label{9ineq0}
						&\frac{d}{dt}\big(\|\Lambda^s (\rho,u)\|^2_{L^2}+\|\Lambda^s g\|^2_{L^2(\mathcal{L}^2)}+\frac{C_2\eta}{(1+t)\ln^2(e+t)}\sum_{j\in\sigma_R}\langle\Lambda^{s}\dot{\Delta}_j\rho,\Lambda^{s-1} \dot{\Delta}_ju\rangle \big)\\ \notag
						& + \frac{C_2\eta\gamma}{(1+t)\ln^2(e+t)}\int_{S^c(R)}|\xi|^{2s}|\widehat{\rho}|^2d\xi + \|\Lambda^{s + 1} u\|^2_{L^2}+ \|\Lambda^{s} g\|^2_{L^2(\mathcal{L}^2)} \\ \notag
						&\lesssim (1+t)^{-1-\delta}\int_{S(t)}|\xi|^{2s}(|\widehat{\rho}|^2 + |\widehat{u}|^2)d\xi + (1+t)^{-1-\delta}\int_{S^c(t)} |\xi|^{2s}|\widehat{\rho}|^2d\xi\\ \notag
						&~~~~+\frac{\eta}{(1+t)^2\ln^2(e+t)}\sum_{j\in\sigma_R}\langle\Lambda^{s}\dot{\Delta}_j\rho,\Lambda^{s-1} \dot{\Delta}_ju\rangle\\ \notag
						&~~~~+\frac{\eta}{(1+t)\ln^2(e+t)}\sum_{j\in\sigma_R}\|\Lambda^{s}\dot{\Delta}_ju\|^2_{L^2}.
					\end{align}
					for some positive constant $\delta$.
				\end{lemm}
			\begin{proof}
			\textbf{Dissipation of $u$ and $g$:}\\
					Applying $\Lambda^s$ to $(\ref{eq1})_1$ and taking inner product with $\Lambda^{s}\rho$, we obtain
				\begin{align}\label{9ineq1}
					\frac{d}{dt}\|\Lambda^s\rho\|^2_{L^2} -\langle\Lambda^su,\Lambda^{s+1}\rho\rangle = \langle\Lambda^{s}(\rho {\rm div}~u),\Lambda^{s}\rho\rangle + \langle\Lambda^{s}(u\cdot\nabla\rho),\Lambda^{s}\rho\rangle.
				\end{align}
				According to Lemma \ref{Lemma2} and Proposition \ref{prop4}, we infer that
				\begin{align}\label{9ineq2}
					\langle\Lambda^{s}(\rho {\rm div}~u),\Lambda^{s}\rho\rangle &\lesssim \left(\|{\rm div}~u\|_{L^{\infty}}\|\Lambda^{s}\rho\|_{L^2} + \|\rho\|_{L^{\infty}}\|\Lambda^{s+1}u\|_{L^2}\right)\|\Lambda^{s}\rho\|_{L^2} \\ \notag
					&\lesssim \varepsilon\|\Lambda^{s+1}u\|^2_{L^2} + (1+t)^{-1-\delta}\left(\|\Lambda^{s}\rho^{high}\|^2_{L^2}+\|\Lambda^{s}\rho^{low}\|^2_{L^2}\right).
				\end{align}
				By virtue of Lemmas \ref{Lemma5} and \ref{2lem2} , we deduce that
				\begin{align}\label{9ineq3}
					\|[\Lambda^{s},u\cdot\nabla]\rho\|_{L^2}\|\Lambda^{s}\rho\|_{L^2} &\lesssim \|\Lambda^su\|_{L^{4}}\|\nabla\rho\|_{L^{4}}\|\Lambda^{s}\rho\|_{L^2} + \|\nabla u\|_{\infty}\|\Lambda^s\rho\|^2_{L^2} \\ \notag
					&\lesssim \varepsilon\|\Lambda^{s+1}u\|^2_{L^2} + (1+t)^{-1-\delta}\left(\|\Lambda^{s}\rho^{high}\|^2_{L^2}+\|\Lambda^{s}\rho^{low}\|^2_{L^2}\right) ,
				\end{align}
				which implies
				\begin{align}\label{9ineq4}
					\langle\Lambda^{s}u\cdot\nabla\rho,\Lambda^{s}\rho\rangle &= \langle u\cdot\nabla\Lambda^{s}\rho,\Lambda^{s}\rho\rangle + \langle[\Lambda^{s},u\cdot\nabla]\rho,\Lambda^{s}\rho\rangle \\ \notag
					&\lesssim \varepsilon\|\Lambda^{s+1}u\|^2_{L^2} + (1+t)^{-1-\delta}\left(\|\Lambda^{s}\rho^{high}\|^2_{L^2}+\|\Lambda^{s}\rho^{low}\|^2_{L^2}\right).
				\end{align}
				Combining estimates \eqref{9ineq1}, \eqref{9ineq2} and \eqref{9ineq4} , we conclude that
				\begin{align}\label{9ineq5}
					\frac{d}{dt}\|\Lambda^s\rho\|^2_{L^2} &-\langle\Lambda^su,\Lambda^{s+1}\rho\rangle \\ \notag
					&\lesssim \varepsilon\|\Lambda^{s+1}u\|^2_{L^2} + (1+t)^{-1-\delta}\left(\|\Lambda^{s}\rho^{high}\|^2_{L^2}+\|\Lambda^{s}\rho^{low}\|^2_{L^2}\right).
				\end{align}
				Similarly, applying $\Lambda^s$ to $(\ref{eq1})_2$ and taking inner product with $\Lambda^{s}u$, we obtain
				\begin{align}\label{9ineq6}
					\frac{d}{dt}\|\Lambda^s u\|^2_{L^2} &+\gamma \langle\Lambda^su,\Lambda^{s+1}\rho\rangle + \|\Lambda^{s+1}u\|^2_{L^2} \\ \notag
					&- \langle\Lambda^s {\rm div}~\tau,\Lambda^{s}u\rangle
					= \langle\Lambda^{s}G,\Lambda^{s}u\rangle,
				\end{align}
				with
				\begin{align}\label{9ineq7}
					\langle\Lambda^{s}G,\Lambda^{s}u\rangle
					& = \langle\Lambda^{s}\left( \frac{\rho}{1+\rho}{\rm div}~\Sigma u + [h(\rho)-\gamma]\nabla\rho\right),\Lambda^{s}u\rangle \\ \notag
					&~~~~+\langle\Lambda^{s}\left( \frac{\rho}{1+\rho}{\rm div}~\tau + u\nabla u \right) ,\Lambda^{s}u\rangle.
				\end{align}
				According to Lemma \ref{2lem2} and Proposition \ref{prop4}, we infer that
				\begin{align}\label{9ineq8}
					\langle\Lambda^{s}\left(\frac{\rho}{1+\rho}{\rm div}~\Sigma u\right),\Lambda^{s}u\rangle &\leq \|\Lambda^{s-1}\left(\frac{\rho}{1+\rho}{\rm div}~\Sigma u\right)\|_{L^2}\|\Lambda^{s+1}u\|_{L^2} \\ \notag
					& \lesssim \left(\|\Lambda^{s-1}\rho\|_{L^{4}}\|\nabla^2u\|_{L^{4}} + \|\rho\|_{L^{\infty}}\|\Lambda^{s+1}u\|_{L^2}\right)\|\Lambda^{s+1}u\|_{L^2} \\ \notag
					&\lesssim \varepsilon\|\Lambda^{s+1}u\|^2_{L^2} + (1+t)^{-1-\delta}\left(\|\Lambda^{s}\rho^{high}\|^2_{L^2}+\|\Lambda^{s}\rho^{low}\|^2_{L^2}\right),
				\end{align}
				and
				\begin{align}\label{9ineq9}
					\langle\Lambda^{s}([h(\rho)-\gamma]\nabla\rho),\Lambda^{s}u\rangle &\leq \|\Lambda^{s-1}([h(\rho)-\gamma]\nabla\rho)\|_{L^2}\|\Lambda^{s+1}u\|_{L^2}\\ \notag
					& \lesssim \left(\|\Lambda^{s-1}\rho\|_{L^{4}}\|\nabla\rho\|_{L^{4}} + \|\rho\|_{L^{\infty}}\|\Lambda^{s}\rho\|_{L^2}\right)\|\Lambda^{s+1}u\|_{L^2} \\ \notag
					&\lesssim \varepsilon\|\Lambda^{s+1}u\|^2_{L^2} + (1+t)^{-1-\delta}\left(\|\Lambda^{s}\rho^{high}\|^2_{L^2}+\|\Lambda^{s}\rho^{low}\|^2_{L^2}\right).
				\end{align}
                Analogously, we have
				\begin{align}\label{9ineq10}
					\langle\Lambda^{s}\left(\frac{\rho}{1+\rho}{\rm div}~\tau\right),\Lambda^{s}u\rangle 
					&\leq \|\Lambda^{s-1}\left(\frac{\rho}{1+\rho}{\rm div}~\tau\right)\|_{L^2}\|\Lambda^{s+1}u\|_{L^2}\\ \notag
					& \lesssim \left(\|\Lambda^{s-1}\rho\|_{L^{4}}\|\nabla\tau\|_{L^{4}} + \|\rho\|_{L^{\infty}}\|\Lambda^{s}\tau\|_{L^2}\right)\|\Lambda^{s+1}u\|_{L^2} \\ \notag
					&\lesssim  (1+t)^{-1-\delta}\left(\|\Lambda^{s}\rho^{high}\|^2_{L^2}+\|\Lambda^{s}\rho^{low}\|^2_{L^2}\right)\\ \notag
					&~~~~+\varepsilon\left(\|\Lambda^{s+1}u\|^2_{L^2} + \|\Lambda^sg\|^2_{L^2(\mathcal{L}^2)}\right) ,
				\end{align}
				and
				\begin{align}\label{9ineq11}
					\langle\Lambda^{s}(u\cdot\nabla u),\Lambda^{s}u\rangle &\leq \|\Lambda^{s-1}(u\cdot\nabla u)\|_{L^2}\|\Lambda^{s+1}u\|_{L^2}\\ \notag
					& \lesssim \left(\|\Lambda^{s-1}u\|_{L^{4}}\|\nabla u\|_{L^{4}} + \|u\|_{L^{4}}\|\Lambda^{s}u\|_{L^4} \right)\|\Lambda^{s+1}u\|_{L^2} \\ \notag
					& \lesssim \varepsilon\|\Lambda^{s+1}u\|^2_{L^2} + (1+t)^{-1-\delta}\|\Lambda^{s}u^{low}\|^2_{L^2}.
				\end{align}
				Hence, together with estimates from \eqref{9ineq8} to \eqref{9ineq11}, one can arrive at
				\begin{align}\label{9ineq12}
					\frac{1}{2}\frac{d}{dt}\|\Lambda^su\|^2_{L^2} &+\mu\|\nabla\Lambda^s u\|^2_{L^2}+(\mu+\mu')\|{\rm div}~\Lambda^su\|^2_{L^2} 
					\\ \notag
					&~~~~+ \gamma\langle\Lambda^s\nabla\rho,\Lambda^su\rangle - \langle\Lambda^s{\rm div}~\tau,\Lambda^su\rangle \\ \notag
					&\lesssim  \varepsilon\left(\|\Lambda^{s+1}u\|^2_{L^2} + \|\Lambda^sg\|^2_{L^2(\mathcal{L}^2)}\right) + (1+t)^{-1-\delta}\|\Lambda^{s}(\rho,u)^{low}\|^2_{L^2}.
				\end{align}
				Applying $\Lambda^s$ to $(\ref{eq1})_3$ and taking $L^2(\mathcal{L}^{2})$ inner product with $\Lambda^{s}g$, we obtain
				\begin{align}\label{9ineq13}
					\frac{d}{dt}\|\Lambda^sg\|^2_{L^2} + \|\nabla_qg\|^2_{L^2(\mathcal{L}^2)} +\langle\Lambda^su,\Lambda^{s}{\rm div}~\tau\rangle 
					= \langle\Lambda^sH,\Lambda^{s}g\rangle,
				\end{align}
				with
				\begin{align}\label{9ineq14}
					\langle\Lambda^{s}H,\Lambda^{s}g\rangle &= -\langle\Lambda^{s}\left(u\cdot\nabla g\right),\Lambda^{s}g\rangle+\langle\Lambda^{s}\left(\nabla u qg\right),\nabla_q\Lambda^{s}g\rangle.
				\end{align}
				According to Lemmas \ref{Lemma4}, \ref{2lem2} and Proposition \ref{prop4} , we infer that
				\begin{align}\label{9ineq15}
					\langle \Lambda^s (u\cdot\nabla g),\Lambda^s g\rangle
					&= \langle u\cdot\nabla\Lambda^{s}g,\Lambda^{s}g\rangle + \langle [\Lambda^{s},u\cdot\nabla]g,\Lambda^{s}g\rangle \\ \notag
					&\lesssim \|\Lambda^{s}u\|_{L^{4}}\|\nabla g\|_{L^4(\mathcal{L}^{2})}\|\Lambda^{s}g\|_{L^2(\mathcal{L}^{2})} + \|\nabla u\|_{L^{\infty}}\|\Lambda^{s}g\|^2_{L^2(\mathcal{L}^{2})} \\ \notag
					& \lesssim \varepsilon\left( \|\Lambda^{s+1}u\|^2_{L^{2}} + \|\nabla_q\Lambda^{s}g\|^2_{L^2(\mathcal{L}^{2})}\right)+ (1+t)^{-1-\delta}\|\Lambda^{s}u^{low}\|^2_{L^2}.
				\end{align}
				Similarly, we obtain from Theorem \ref{th1} that
				\begin{align}\label{9ineq16}
					\langle\Lambda^{s}\left(\nabla u qg\right),\nabla_q\Lambda^{s}g\rangle
					&\lesssim\|\langle q \rangle g\|_{H^{s-1}(\mathcal{L}^{2})}\|\Lambda^{s+1} u\|_{L^2}\|\nabla_q\Lambda^s g\|_{L^2(\mathcal{L}^{2})}\\ \notag
					&~~~~+\|\nabla u\|_{L^{\infty}}\|q \Lambda^sg\|_{L^2(\mathcal{L}^{2})}\|\nabla_q\Lambda^s g\|_{L^2(\mathcal{L}^{2})}\\ \notag
					& \lesssim \varepsilon\left( \|\Lambda^{s+1}u\|^2_{L^{2}} + \|\nabla_q\Lambda^{s}g\|^2_{L^2(\mathcal{L}^{2})}\right).
				\end{align}
				Therefore, we deduce from \eqref{9ineq13} to \eqref{9ineq16} that
				\begin{align}\label{9ineq17}
					\frac {1} {2}&\frac {d} {dt} \|\Lambda^s g\|^2_{L^2(\mathcal{L}^{2})}+\|\nabla_q  \Lambda^s g\|^2_{L^2(\mathcal{L}^{2})}+\int_{\mathbb{R}^{d}}\nabla\Lambda^{s}u:\Lambda^{s}\tau dx \\ \notag
					&\lesssim \varepsilon( \|\Lambda^{s+1}u\|^2_{L^{2}} + \|\nabla_q\Lambda^{s}g\|^2_{L^2(\mathcal{L}^{2})})+ (1+t)^{-1-\delta}\|\Lambda^{s}u^{low}\|^2_{L^2}.
				\end{align}
				Together with estimates \eqref{9ineq5}, \eqref{9ineq12} and \eqref{9ineq17}, we conclude that
				\begin{align}\label{9ineq18}
					&\frac{d}{dt}\left(\gamma\|\Lambda^s\rho\|^2_{L^2}+\|\Lambda^s u\|^2_{L^2}+\|\Lambda^{s}g\|^2_{L^2(\mathcal{L}^2)}\right) \\ \notag
					&~~~~+\mu\|\nabla\Lambda^s u\|^2_{L^2}+(\mu+\mu')\|{\rm div}~\Lambda^su\|^2_{L^2} + \|\nabla_q\Lambda^sg\|^2_{L^2(\mathcal{L}^2)} \\ \notag
					&\lesssim (1+t)^{-1-\delta}\left(\|\Lambda^{s}\rho^{high}\|^2_{L^2} + \|\Lambda^{s}(\rho,u)^{low}\|^2_{L^2}\right).
				\end{align}
				\textbf{Dissipatioin of $\rho$ in high frequency :}\\
				Applying $\Lambda^{s}\dot{\Delta}_j$ to $(\ref{eq1})_1$ and taking $L^2$ inner product with $\Lambda^{s-1}\dot{\Delta}_ju$, we have
				\begin{align}\label{9ineq19}
					\partial_t \langle\Lambda^{s}\dot{\Delta}_j\rho,\Lambda^{s-1}\dot{\Delta}_ju\rangle - \|\Lambda^{s}\dot{\Delta}_ju\|^2_{L^2} = \langle \Lambda^{s}\dot{\Delta}_jF,\Lambda^{s-1}\dot{\Delta}_ju\rangle.
				\end{align}
				We infer from Lemma \ref{2lem2} and proposition \ref{prop4} that
				\begin{align}\label{9ineq20}
					\langle \Lambda^{s}\dot{\Delta}_jF,\Lambda^{s-1}\dot{\Delta}_ju\rangle &=  - \langle \Lambda^{s}\dot{\Delta}_j(\rho u),\Lambda^{s}\dot{\Delta}_ju\rangle \\ \notag
					& \leq \|\Lambda^{s-1}\dot{\Delta}_j(\rho u)\|_{L^2}\|\Lambda^{s+1}\dot{\Delta}_ju\|_{L^2} \\ \notag
					&\lesssim  d_j\|\Lambda^{s+1}u\|^2_{L^2} + d_j \left( \|\Lambda^{s-1}\rho\|^2_{L^2}\|u\|^2_{L^{\infty}} + \|\Lambda^{s-1}u\|^2_{L^2}\|\rho\|^2_{L^{\infty}}\right) \\ \notag
					&\lesssim  d_j\|\Lambda^{s+1}u\|^2_{L^2} + d_j  \|\rho\|^{\frac{2}{s}}_{L^2}\|u\|^{2-\frac{2}{s}}_{L^{2}}\|\Lambda^s\rho\|^{2-\frac{2}{s}}_{L^{\infty}}\|\Lambda^{s}u\|^{\frac{2}{s}}_{L^2} \\ \notag
					&~~~~+d_j  \|\rho\|^{2-\frac{2}{s}}_{L^2}\|u\|^{\frac{2}{s}}_{L^{2}}\|\Lambda^s\rho\|^{\frac{2}{s}}_{L^{\infty}}\|\Lambda^{s}u\|^{2-\frac{2}{s}}_{L^2} \\ \notag
					&\lesssim d_j\|\Lambda^{s+1}u\|^2_{L^2} + d_j(1+t)^{-1}  \left(\|\Lambda^{s}\rho^{high}\|^2_{L^2} + \|\Lambda^{s}(\rho,u)^{low}\|^2_{L^2}\right)
				\end{align}
				for some $\{d_j\}_{j\in Z}\in l^1$. Therefore
				\begin{align}\label{9ineq21}
					\langle\partial_t \Lambda^{s}\dot{\Delta}_j\rho,\Lambda^{s-1}\dot{\Delta}_ju\rangle &- \|\Lambda^{s}\dot{\Delta}_ju\|^2_{L^2} \lesssim d_j\|\Lambda^{s+1}u\|^2_{L^2} \\ \notag
					&+ d_j(1+t)^{-1}\left(\|\Lambda^{s}\rho^{high}\|^2_{L^2} + \|\Lambda^{s}(\rho,u)^{low}\|^2_{L^2}\right).~~~~~
				\end{align}
			Similarly, applying $\Lambda^{s-1}\dot{\Delta}_j$ to $\eqref{eq1}_2$ and taking $L^2$ inner product with $\Lambda^{s}\dot{\Delta}_j\rho$ leads to
				\begin{align}\label{9ineq22}
					\langle \Lambda^{s}\dot{\Delta}_j\rho,\partial_t\Lambda^{s-1}\dot{\Delta}_ju\rangle &+ \gamma\|\Lambda^s\dot{\Delta}_j\rho\|^2_{L^2} \lesssim d_j \left( \|\Lambda^{s+1}u\|^2_{L^2} + \|\Lambda^{s}g\|^2_{L^2(\mathcal{L}^2)} \right) \\ \notag
					&~~~~+ d_j(1+t)^{-\frac{1}{2}}\left(\|\Lambda^{s}\rho^{high}\|^2_{L^2}+\|\Lambda^{s}(\rho,u)^{low}\|^2_{L^2}\right).
				\end{align}
				Adding up \eqref{9ineq21} and \eqref{9ineq22} , we conclude that
				\begin{align}\label{9ineq23}
					\frac{d}{dt}\langle \Lambda^{s}\dot{\Delta}_j\rho,\Lambda^{s-1}\dot{\Delta}_ju\rangle &+ \gamma\|\Lambda^s\dot{\Delta}_j\rho\|^2_{L^2} - \|\Lambda^{s}\dot{\Delta}_ju\|^2_{L^2} \\ \notag
					&\lesssim d_j\left( \|\Lambda^{s+1}u\|^2_{L^2} + \|\Lambda^{s}g\|^2_{L^2(\mathcal{L}^2)}\right) \\ \notag
					&~~~~+ d_j(1+t)^{-\frac{1}{2}}\left(\|\Lambda^{s}\rho^{high}\|^2_{L^2}+\|\Lambda^{s}(\rho,u)^{low}\|^2_{L^2}\right),
				\end{align}
				which implies
				\begin{align}\label{9ineq24}
					\frac{d}{dt}&\left(\frac{C_2\eta}{(1+t)\ln^2(e+t)}\langle \Lambda^{s}\dot{\Delta}_j\rho,\Lambda^{s-1}\dot{\Delta}_ju\rangle\right) + \frac{C_2\gamma\eta}{(1+t)\ln^2(e+t)}\|\Lambda^s\dot{\Delta}_j\rho\|^2_{L^2}  \\ \notag
					&\lesssim d_j\eta\left( \|\Lambda^{s+1}u\|^2_{L^2} + \|\Lambda^{s}\tau\|^2_{L^2}\right) + d_j\eta(1+t)^{-\frac{3}{2}}\left(\|\Lambda^{s}\rho^{high}\|^2_{L^2}+\|\Lambda^{s}(\rho,u)^{low}\|^2_{L^2}\right) \\ \notag
					&~~~~+ \frac{\eta}{(1+t)^{2}\ln^2(e+t)}\langle \Lambda^{s}\dot{\Delta}_j\rho,\Lambda^{s-1}\dot{\Delta}_ju\rangle +\frac{\eta}{(1+t)\ln^2(e+t)}\|\Lambda^{s}\dot{\Delta}_ju\|^2_{L^2}.
				\end{align}
				Adding $j\in \sigma_R$ up leads to
				\begin{align}\label{9ineq25}
					\frac{d}{dt}&\left(\frac{C_2\eta}{(1+t)ln^2(e+t)}\mathop{\Sigma}\limits_{j\in\sigma_R}\langle \Lambda^{s}\dot{\Delta}_j\rho,\Lambda^{s-1}\dot{\Delta}_ju\rangle\right) + \frac{C_2\gamma\eta}{(1+t)ln^2(e+t)}\int_{S^c(R)}|\xi|^{2s}|\widehat{\rho}|^2d\xi \\ \notag
					&\lesssim \eta\left( \|\Lambda^{s+1}u\|^2_{L^2} + \|\Lambda^{s}\tau\|^2_{L^2}\right)
					+ \eta(1+t)^{-\frac{3}{2}}\left(\|\Lambda^{s}\rho^{high}\|^2_{L^2}+\|\Lambda^{s}(\rho,u)^{low}\|^2_{L^2}\right) \\ \notag
					&~~~~+\frac{\eta}{(1+t)^{2}\ln^2(e+t)}\mathop{\Sigma}\limits_{j\in\sigma_R}\langle \Lambda^{s}\dot{\Delta}_j\rho,\Lambda^{s-1}\dot{\Delta}_ju\rangle +\frac{\eta}{(1+t)\ln^2(e+t)}\mathop{\Sigma}\limits_{j\in\sigma_R}\|\Lambda^{s}\dot{\Delta}_ju\|^2_{L^2}.
				\end{align}
				Combining with estimates \ref{9ineq18} and \ref{9ineq25}, we thus complete the proof of Lemma \ref{2lem2}. 
			\end{proof}
				\textbf{The proof of Theorem \ref{th2}:}\\
					According to Schonbek's strategy and Lemma \ref{2lem2}, one can arrive at
				\begin{align}\label{10ineq0}
					\frac{d}{dt'}&\left(\|\Lambda^s (\sqrt{\gamma}\rho,u)\|^2_{L^2} + \|\Lambda^s g\|^2_{L^2(\mathcal{L}^2)}+\frac{C_2\eta}{(1+t')\ln^2(e+t')}\mathop{\Sigma}\limits_{j\in\sigma_R}\langle\Lambda^{s}\dot{\Delta}_j\rho,\Lambda^{s-1} \dot{\Delta}_ju\rangle \right)\\ \notag
					&~~~~+ \frac{C_2}{1+t'}\int|\xi|^{2s}|\widehat{u}|^2d\xi + \frac{C_2\eta\gamma}{(1+t')\ln^2(e+t')}\int|\xi|^{2s}|\widehat{\rho}|^2d\xi + \|\Lambda^{s} g\|^2_{L^2(\mathcal{L}^2)} \\ \notag
					&\lesssim \frac{C_2}{1+t'}\int_{S(t')}|\xi|^{2s}|\widehat{u}|^2d\xi + \frac{\eta}{(1+t')\ln^2(e+t')}\int_{S(R)} |\xi|^{2s}\left(|\widehat{\rho}|^2 + |\widehat{u}|^2\right)d\xi \\ \notag
					&~~~~+(1+t')^{-1-\delta}\int_{S^c(t')}|\xi|^{2s}|\widehat{\rho}|^2d\xi+\frac{\eta}{(1+t')^2\ln^2(e+t')}\mathop{\Sigma}\limits_{j\in\sigma_R}\langle\Lambda^{s}\dot{\Delta}_j\rho,\Lambda^{s-1} \dot{\Delta}_ju\rangle.
				\end{align}
				Set a positive constant $T_2$ sufficiently large. By virtue of Proposition \ref{prop4}, we deduce that
				\begin{align}\label{10ineq1}
					\frac{d}{dt'}&\left(\|\Lambda^s (\sqrt{\gamma}\rho,u)\|^2_{L^2} + \|\Lambda^s g\|^2_{L^2(\mathcal{L}^2)}+\frac{C_2\eta}{(1+t')\ln^2(e+t')}\mathop{\Sigma}\limits_{j\in\sigma_R}\langle\Lambda^{s}\dot{\Delta}_j\rho,\Lambda^{s-1} \dot{\Delta}_ju\rangle \right)\\ \notag
					&~~~~+ \frac{C_2}{1+t'}\int|\xi|^{2s}|\widehat{u}|^2d\xi + \frac{C_2\eta\gamma}{(1+t')\ln^2(e+t')}\int|\xi|^{2s}|\widehat{\rho}|^2d\xi + \|\Lambda^{s} g\|^2_{L^2(\mathcal{L}^2)} \\ \notag
					&\lesssim \frac{C_2}{(1+t')^{s+2}} + \frac{\eta}{(1+t')^{2}\ln^2(e+t')(1+R)^{s}}+\frac{\eta(1+R)}{(1+t')^2\ln^2(e+t')}\|\Lambda^{s}u\|^2_{L^2}.~~~~~~~
				\end{align}
				for any $t'\in [T_2,\infty)$. Multiplying \eqref{10ineq1} by $(1+t')^{s+3}$ and integrating $t'$ over $[T_d,t]$ leads to
				\begin{align}\label{10ineq2}
					&(1+t)^{s+\frac{d}{2}+2}\left(\|\Lambda^s (\sqrt{\gamma}\rho,u)\|^2_{L^2} + \|\Lambda^sg\|^2_{L^2(\mathcal{L}^2)}+\frac{C\eta}{(1+t)\ln^2(e+t)}\mathop{\Sigma}\limits_{j\in\sigma_R}\langle\Lambda^{s}\dot{\Delta}_j\rho,\Lambda^{s-1} \dot{\Delta}_ju\rangle \right)\\ \notag
					&~~~~\lesssim C_0 + C_d(1+t)^2 +\frac{C\eta(1+t)^{s+2}}{(1+R)^{s}} + \int_0^t \frac{\eta(1+R)}{\ln^2(e+t')}(1+t')^{s+1}\|\Lambda^{s}u\|^2_{L^2} dt',
				\end{align}
				where $C_0 = C_{\gamma}\left(\|(\rho_0,u_0)\|^2_{H^s} + \|g_0\|^2_{H^s(\mathcal{L}^2)}+\|\langle q\rangle g_0\|^2_{H^{s-1}(\mathcal{L}^2)}\right)$. Since
				$$\frac{C_2}{1+t}\|\Lambda^{s-1}u^{high}\|^2_{L^2} \leq \|\Lambda^{s}u^{high}\|^2_{L^2},$$
				it follows that
				\begin{align}\label{10ineq3}
					\|\Lambda^s (\sqrt{\gamma}\rho,u)\|^2_{L^2} + \|\Lambda^s g\|^2_{L^2(\mathcal{L}^2)}&+\frac{C_2\eta}{(1+t)ln^2(e+t)}\mathop{\Sigma}\limits_{j\in\sigma_t}\langle\Lambda^{s}\dot{\Delta}_j\rho,\Lambda^{s-1} \dot{\Delta}_ju\rangle \\ \notag
					&\geq \frac{1}{2}\left(\|\Lambda^s (\sqrt{\gamma}\rho,u)\|^2_{L^2} + \|\Lambda^s g\|^2_{L^2(\mathcal{L}^2)}\right).
				\end{align}
				for some $\eta$ small enough. Therefore, taking $R=t$, we obtian
				\begin{align}\label{10ineq4}
					(1+t)^{s+3}&\left(\|\Lambda^s (\sqrt{\gamma}\rho,u)\|^2_{L^2} + \|\Lambda^s g\|^2_{L^2(\mathcal{L}^2)}\right) \lesssim C_0 + (1+t)^2 \\ \notag
					&+ (1+t)\int_0^t\frac{C\eta}{\ln^2(e+t')}(1+t')^{s+1}\|\Lambda^{s}u\|^2_{L^2}dt'.
				\end{align}
				Denote ${\rm M}(t)=\mathop{\sup}\limits_{t'\in[0,t]}(1+t')^{s+1}\left(\|\Lambda^s (\sqrt{\gamma}\rho,u)\|^2_{L^2} + \|\Lambda^s g\|^2_{L^2(\mathcal{L}^2)}\right)$, we deduce that
				\begin{align}\label{10ineq5}
					{\rm M}(t) \lesssim C_0 + 1 + \int_0^t\frac{{\rm M}(t')}{(1+t')\ln^2(e+t')}dt',
				\end{align}
				which implies
				\begin{align}\label{10ineq6}
					\|\Lambda^s (\rho,u)\|^2_{L^2} + \|\Lambda^s g\|^2_{L^2(\mathcal{L}^2)} \lesssim (1+t)^{-s-1}.~~~
				\end{align}
				Furthermore, taking $L^2(\mathcal{L}^2)$ innner product with $g$ to $\eqref{eq1}_3$, we have
				\begin{align}\label{10ineq7}
					\frac {d} {dt} \|g\|^2_{L^2(\mathcal{L}^{2})}+\|g\|^2_{L^2(\mathcal{L}^{2})} \lesssim \|\nabla u\|^2_{L^2},~~~
				\end{align}
				then we infer by using Duhamel's principle that
				\begin{align}\label{10ineq8}
					\|g\|^2_{L^2(\mathcal{L}^{2})} &\lesssim e^{-t}\|g_0\|^2_{L^2(\mathcal{L}^{2})} + \int_0^t e^{-(t-t')}\|\nabla u\|^2_{L^2}dt'\\ \notag
					&\lesssim e^{-t}\|g_0\|^2_{L^2(\mathcal{L}^{2})} + \int_0^t e^{-(t-t')}(1+t')^{-2}dt'\\ \notag
					&\lesssim (1+t)^{-2}.
				\end{align}
				Applying $\Lambda^\sigma$ to $(\ref{eq1})_3$ and taking $L^2(\mathcal{L}^{2})$ inner product with $\Lambda^{\sigma}g$ leads to
				\begin{align}\label{10ineq9}
					&\frac {1} {2}\frac {d} {dt} \|\Lambda^\sigma g\|^2_{L^2(\mathcal{L}^{2})}+\|\nabla_q  \Lambda^\sigma g\|^2_{L^2(\mathcal{L}^{2})}+\int_{\mathbb{R}^{d}}\nabla\Lambda^{\sigma}u:\Lambda^{\sigma}\tau dx
					=-\langle \Lambda^\sigma (u\cdot\nabla g),\Lambda^\sigma g\rangle_M  \\ \notag
					&-\langle\frac 1 {\psi_\infty} \nabla_q\cdot(\Lambda^\sigma\nabla uqg\psi_\infty),\Lambda^\sigma g\rangle_M
					-\langle\frac 1 {\psi_\infty} \nabla_q\cdot(q\psi_\infty[\Lambda^\sigma,g]\nabla u),\Lambda^\sigma g\rangle_M.
				\end{align}
				According to Lemmas \ref{Lemma3}, \ref{Lemma4}, Proposition \ref{prop4} and \eqref{10ineq6}, we deduce that
				\begin{align}\label{10ineq10}
					\|\Lambda^\sigma g\|^2_{L^2(\mathcal{L}^{2})}
					&\lesssim e^{-t}\|g_0\|^2_{\dot{H}^\sigma(\mathcal{L}^{2})}
					+ \int_0^t e^{-(t-t')}\|\Lambda^\sigma u\|^2_{L^2}\|\nabla g\|^2_{H^{s-1}(\mathcal{L}^{2})}dt' \\ \notag
					&+ \int_0^t e^{-(t-t')}(\|\langle q\rangle g\|^2_{H^{s-1}(\mathcal{L}^{2})}+1)\|\Lambda^{\sigma+1}u\|^2_{L^2}dt'\\ \notag
					& \lesssim  e^{-t}\|g_0\|^2_{\dot{H}^\sigma(\mathcal{L}^{2})} + \int_0^t e^{-(t-t')}((1+t')^{ - \sigma - 2} + (1+t')^{- \sigma - 3})  dt'\\ \notag
					& \lesssim e^{-t}\|g_0\|^2_{\dot{H}^\sigma(\mathcal{L}^{2})} + (1+t)^{ - \sigma - 2},
				\end{align}
				which implies 
				$$
				\|g\|_{\dot{H}^\sigma(\mathcal{L}^{2})} \lesssim (1+t)^{ - \frac{\sigma}{2}-1},
				$$
				for any $\sigma\in[0,s-1]$. We thus complete the proof of Theorem \ref{th2} by the interpolation.
				\hfill$\Box$\\
				\smallskip
				\noindent\textbf{Acknowledgments} This work was
				partially supported by the National key R\&D Program of China(2021YFA1002100), the National Natural Science Foundation of China (No.12171493), the Macao Science and Technology Development Fund (No. 098/2013/A3), and Guangdong Province of China Special Support Program (No. 8-2015),
				and the key project of the Natural Science Foundation of Guangdong province (No. 2016A030311004).
				
				%The authors thank the referee for valuable comments and suggestions.
				
				\phantomsection
				\addcontentsline{toc}{section}{\refname}
				%添加参考文献到书签，宏包 hyperref
				\bibliographystyle{abbrv} %plain ,%alpha, %abbrv
				\bibliography{hookeanref2d}

			\end{document}